\documentclass[letterpaper, 10pt, conference]{ieeeconf}      

\IEEEoverridecommandlockouts                              

\usepackage[utf8]{inputenc}
\usepackage{amssymb}
\usepackage{amsmath}
\usepackage{bm}
\usepackage{hyperref}
\usepackage{tikz}
\usepackage{tkz-euclide}
\usepackage{chngcntr}
\usepackage{verbatim}
\usepackage{mathtools}
\usepackage{esvect}
\usepackage{subfiles}
\usepackage{color}
\usepackage{xspace}
\usepackage{xparse}
\usepackage[noend]{algpseudocode}
\usepackage{subfigure}
\usepackage{bbm}
\usepackage{enumerate}

\usepackage{nicefrac}

\usepackage{atbegshi}
\usepackage{array, tabularx, caption, boldline}
\usepackage{url}
\usepackage[misc,geometry]{ifsym} 
\let\labelindent\relax
\usepackage{enumitem}
\usepackage{graphicx}
\usepackage{color}
\usepackage{cite}
\usepackage{graphicx}
\usepackage{algorithm}

\usepackage{makecell}
\usepackage{multirow}
\usepackage{booktabs}
\def\R{\mathbb{R}}
\newcommand{\E}{{\mathbb E}}
\def\W{\mathcal W}

\def\R{\mathbb R}
\def\E{\mathbb E}

\def\e{\varepsilon}
\def\la{\langle}
\def\ra{\rangle}
\def\vp{\varphi}
\def\y{\mathbf{y}}
\def\x{\mathbf{x}}

\newcommand{\eqdef}{\overset{\text{def}}{=}}

\newcommand{\argmin}{\mathop{\arg\!\min}}
\newcommand{\argmax}{\mathop{\arg\!\max}}

\def\Blm{\boldsymbol{\lambda}}
\def\Bzeta{\boldsymbol{\zeta}}
\def\Bxi{\boldsymbol{\xi}}

\def\Blm{\boldsymbol{\lambda}}
\def\Bxi{\boldsymbol{\xi}}
\def\p{\mathtt{p}}

\def\lm{\lambda}

\def\a{\mathbf{a}}
\def\EE{\mathbb E}
\def\PP{\mathbb P}

\usepackage[percent]{overpic}

\newtheorem{theorem}{Theorem}
\newtheorem{lemma}[theorem]{Lemma}

\newtheorem{corollary}[theorem]{Corollary}

\allowdisplaybreaks

\def\pd#1{{#1}} 
\def\pdd#1{{#1}} 
\def\dd#1{{#1}} 
\def\eg#1{{#1}} 
\def\ed#1{{#1}} 

\usepackage[colorinlistoftodos,textsize=scriptsize]{todonotes}

\title{\LARGE \bf
On Primal--Dual Approach for Distributed Stochastic Convex Optimization over Networks
}

\author{Darina Dvinskikh, Eduard Gorbunov, Alexander Gasnikov, Pavel Dvurechensky, C\'esar A. Uribe
	\thanks{The work of D. Dvinskikh and P. Dvurechensky was funded by Russian Science Foundation (project 18-71-10108). The work  of E. Gorbunov was supported by RFBR 18-31-20005 mol-a-ved. The work of A. Gasnikov was supported by RFBR 18-29-03071 mk.}
	\thanks{D.D and P.D. are with the Weierstrass Institute for Applied Analysis and Stochastics, Germany, and the Institute for Information Transmission Problems, Russia 	(\textit{\{darina.dvinskikh,pavel.dvurechensky\}@wias-berlin.de}). E.G. is with the Moscow Institute of Physics and Technology, Russia (\textit{eduard.gorbunov@phystech.edu}). A.G.  is with \pd{Moscow Institute of Physics and Technology}, Institute for Information Transmission Problems, Russia and National Research University Higher School of Economics, Russia (\textit{gasnikov@yandex.ru}). C.A.U. is with the  the Laboratory for Information and Decision Systems (LIDS), Massachusetts Institute of Technology, USA (\textit{cauribe@mit.edu}). }%
}

\begin{document}
\maketitle
\thispagestyle{empty}
\pagestyle{empty}

\begin{abstract}

We introduce a primal-dual stochastic gradient oracle method for distributed convex optimization problems over networks. We show that the proposed method is optimal in terms of communication steps. Additionally, we propose a new analysis method for the rate of convergence in terms of duality gap and probability of large deviations. This analysis is based on a new technique that allows to bound the distance between the iteration sequence and the optimal point. By the proper choice of batch size, we can guarantee that this distance equals (up to a constant) to the distance between the starting point and the solution.

\end{abstract}

\section{Introduction}
\label{S:Intro}

Distributed algorithms have been prevalent in the control theory and machine learning communities since early 70s and 80s~\cite{bor82,tsi84,deg74}. The structural flexibilities introduced by a networked structure has been particularly relevant for recent applications, such as robotics and resource allocation~\cite{xia06,rab04,kra13,ned17e,ivanova2018composite}, where large quantities of data are involved, and generation and processing of information is not centralized~\cite{bot10,boy11,aba16,ned16w,ned15}.

A distributed system is usually modeled as a network of computing agents connected in a definite way. These agents can act as local processors or sensors, and have communication capabilities to exchange information with each other. Precisely, the communication between agents is subject to the constraints imposed by the network structure. The object of study of distributed optimization is then to design algorithms that can be locally executed by the agents, and that exploit the network communications to solve a network-wide global problem cooperatively~\cite{ned09, ram10}. 

Formally, we consider the optimization problem of minimizing the finite sum of $m$ convex functions
        \begin{equation}\label{eq:distr_opt}
        \min_{x\in \R^n} f(x) :=\sum_{i=1}^m f_i(x), 
    \end{equation}
    where each agent $i = \{1,2,\dots,m\}$ in the network has access to the function $f_i$ only, and yet, we seek that every agent cooperatively achieves a solution of~\eqref{eq:distr_opt}. 
    
    In this paper,  we consider the stochastic version of  problem~\eqref{eq:distr_opt}, when $f_i(x) = \E \tilde{f}_i (x, \xi)$, and $\xi$ is a random variable. We provide an accelerated dual gradient method for this stochastic problem and estimate the number of communication steps in the network and the number of stochastic oracle calls in order to obtain a solution with high probability. 
    
    Optimal methods for distributed optimization over networks were recently proposed and analyzed~\cite{scaman2017optimal,uribe2018dual}. However, there were only studied for deterministic settings. In~\cite{lan2017communication}, the authors studied a primal-dual method for stochastic problems. The setting of the latter paper is close to what we consider as the primal approach, but our algorithm and analysis are different, \pd{and, unlike \cite{lan2017communication}, we consider smooth primal problem.} Other approaches for distributed stochastic optimization has been studied in the literature~\cite{Jakovetic2018,li2017distributed}. In contrast, we provide optimal communication complexities, as well as explicit dependency on the network topology. We want to mention that primal approaches were recently studied in \cite{dvinskikh2019decentralized,gorbunov2019optimal}.

\noindent \textbf{Notation:}
We define the maximum eigenvalue and minimal non-zero eigenvalue of a symmetric matrix $W$ as $\lambda_{\max}(W)$ and ${\lm}^+_{\min}(W)$ respectively, and define the condition number of matrix $W$ as $\chi(W)$.
 We denote by $\boldsymbol{1}_m$ the vector of ones in $\mathbb{R}^m$. Denoting by $\|\cdot\|_2$ the standard Euclidean norm, we say that a function $f$ is $M$-Lipschitz if $\|\nabla f(x)\|_2\leq M$, a function $f$ is $L$-smooth if $\|\nabla f(x) - \nabla f(y) \|_2 \leq L\|x-y\|_2$, a function $f$ is $\mu$-strongly convex ($\mu$-s.c.) if, for all $x, y \in \R^n$, $f(y) \geq f(x) + \la \nabla f(x) ,y-x\ra + \frac{\mu}{2}\|x-y\|_2^2$. \pd{Given $\beta \in (0,1)$, we denote $\rho_\beta = 1 + \ln ({1}/{\beta}) + \sqrt{\ln({1}/{\beta})}$.}

\section{Dual distributed approaches}\label{sec:dual}
In this section, we follow \cite{scaman2017optimal,scaman2018optimal,maros2018panda,uribe2018dual} \pd{and use primal-dual accelerated gradient methods \cite{dvurechensky2016primal-dual,chernov2016fast,anikin2017dual,dvurechensky2018computational,guminov2019accelerated}}, and use a dual formulation of the distributed optimization problem to design \dd{a} class \dd{of} optimal algorithms that can be executed over a network. 
Consider a network of $m$ agents whose interactions are represented by a connected and undirected graph $G=(V,E)$ with the set $V$ of $m$ vertices and the set of edges $E = \{(i,j): i,j \in V\}$. Thus, agent $i$ can communicate with agent $j$ if and only if $(i,j)\in E$. \dd{Assume that each agent $i$ has its own vector  vector $y_i^0\in R^n$, and its goal is to find an approximation to the vector $y^* = \frac{{1}}{m}\sum_{i=1}^my_i^0$  by performing communications with  neighboring agents.}
\pd{To do this}, consider the Laplacian of the graph $G$, to be defined as a matrix $\bar{W}$ with entries,
\begin{align*}
[\bar{W}]_{ij} = \begin{cases}
-1,  & \text{if } (i,j) \in E,\\
\text{deg}(i), &\text{if } i= j, \\
0,  & \text{otherwise,}
\end{cases}
\end{align*}
 where ${\rm deg}(i)$ is the degree of vertex $i$ (i.e., the number of neighboring nodes). \pd{Let us denote $W = \bar{W} \otimes I_n$, where $\otimes$ denotes Kronecker product and $I_n$ is the unit matrix.}

\pd{First}, we present the dual formulation of the distributed optimization problem for the deterministic case, and then we develop our novel analysis for the case of  stochastic dual oracles. 

We assume that for all $i=1,\dots, m$ function $f_i$ can be represented as \dd{the} Fenchel-Legendre transform
\[
f_i(x) = \max_{y\in \R^n}\{ \la y, x\ra - \vp_i(y)\}.
\]
Thus, we rewrite the problem \eqref{eq:distr_opt} as follows
\begin{align}\label{eq:fin_sum_rewrite}
     \max_{\substack{x_1,\dots, x_m \in \R^n, \\ x_1=\dots=x_m }}  
     -F(\x):&=- \sum\limits_{i=1}^{m} f_i(x_i)\notag\\
     &= \max_{\substack{x_1,\dots, x_m \in \R^n, \\ \sqrt{W}\x=0 }}  - \sum\limits_{i=1}^{m} f_i(x_i),
\end{align}
    where  $\x = [x_1, \dots, x_m]^T \in \R^{nm}$ is the stacked column vector.
    
Then, we introduce the Lagrangian dual problem to problem \eqref{eq:fin_sum_rewrite} with dual variables   $\y = [y_1^T,\cdots,y_m^T]^T \in \R^{mn}$ as 
\begin{align}\label{eq:DualPr}
&\min_{\y \in \R^{mn}} ~\max_{\x \in\R^{nm} } ~  \sum\limits_{i=1}^{m} \left(\langle y_i, [\sqrt{W}\x]_i\rangle - f_i(x_i)\right)\notag \\
&= \min_{\y \in \R^{mn}} \psi(\y) := \vp(\sqrt{W}\y):= \sum_{i=1}^{m}
\vp_i([\sqrt{W}\y]_i),
\end{align}
 where we used the notations $[\sqrt{W}\x]_i$ and $[\sqrt{W}\y]_i$ for describing the $i$-th $n$-dimensional block of vectors $\sqrt{W}\x$ and $\sqrt{W}\y$ respectively, and also we used the equality $\sum_{i=1}^{m} \langle \y_i, [\sqrt{W}\x]_i\rangle = \sum_{i=1}^{m} \langle [\sqrt{W} \y]_i, \x_i\rangle$.

Note that dealing with the dual problem does not oblige us to use dual oracle of $\nabla \vp_i$. Indeed, 
\begin{align}
    \nabla \vp([\sqrt{W}\y]_i) = [\sqrt{W}\x(\sqrt W\y)]_i,  
\end{align}
where
$   x_i([ W\y]_i) = \argmax\limits_{x_i \in \R^n}  \left\{ \la [\sqrt W\x]_i,y_i\ra - f_i(x_i) \right\}$. So we can use the primal oracle $\nabla f_i$ to solve this auxiliary subproblem and find an approximation to $\nabla \vp_i$.

Making the change of variables  $  \bar \y := \sqrt{W}\y $ and structure of Laplacian matrix $W$ allows us to present accelerated gradient method in a distributed manner for \pd{the} dual problem.

\begin{algorithm}[t]
\caption{Distributed Dual  Algorithm}
\label{Alg:DualNFGM}   
\begin{algorithmic}[1]
   \Require Starting point $\bar \Blm^0 = \bar \y^0=\bar \Bzeta^0= {\x}^0= 0$, number of iterations $N$, $C_0=\alpha_0=0$.
               \State \dd{Each agent $i$ do}
                \For{$k=0,\dots, N-1$}
                \State $\alpha_{k+1} = \frac{k+2}{4L}$, $A_{k+1} = \sum_{i=1}^{k+1}\alpha_i$  
                \State  $\bar \lambda^{k+1}_i = {(\alpha_{k+1}\bar \zeta^k_i + A_k \bar y_i^k)}/{A_{k+1}}.$
                \State
             $\bar \zeta^{k+1}_i= \bar \zeta^{k}_i - \alpha_{k+1} \sum_{j=1}^m W_{ij} x_j (\bar \lambda_j^t).$
                \State $\bar  y_i^{k+1} ={(\alpha_{k+1}\bar \zeta_i^{k+1} + A_k \bar y^k_i)}/{A_{k+1}}.$
      \EndFor
         \State $
                    x^{N}_i = \frac{1}{A_{N}}\sum_{k=0}^{N} \alpha_k x_i(\bar \lambda_i^k).\notag
 $
          
    \Ensure${\x}^{N}$, $\bar \y^{N}$.    
\end{algorithmic}
 \end{algorithm}

\begin{theorem}\label{thm:3}
Let $\varepsilon>0$ be a desired accuracy and assume that $\|\nabla F(\x^*)\|_2 = M_F$ and that the primal objective in \eqref{eq:fin_sum_rewrite} is $\mu$-strongly convex. Then the sequences $ \x^N$ and $\y^N$ generated by Algorithm \ref{Alg:DualNFGM} after
$N=O\big(\sqrt{({M^2_F}/{\mu\e})\chi(W)}\big) $
iterations and oracle calls of dual function $\nabla \vp_i$ per node $i=1,\dots m$  satisfy the following condition
$F(\x^N)+\psi(\bar \y^N) \leq \e$
\end{theorem}


Next, we focus on the case where we only have access to the stochastic dual oracle.

\subsection{Dual Approach with Stochastic Dual Oracle}
\ed{In this section we will assume that the dual function $\varphi(\y) \eqdef \max_{\x\in \R^{mn}}\left\{\la\y, \x \ra - F(\x)\right\}$ could be represented as an expectation of differentiable in $\y$ functions $\varphi(\y,\xi)$, i.e.\ $\varphi(\y) = \EE_{\xi}\left[\varphi(\y, \xi)\right]$. It implies that $\varphi(\sqrt{W}\y) \eqdef \psi(\y) = \EE_{\xi}[\psi(\y,\xi)]$, where $\psi(\y,\xi) \eqdef \varphi(\sqrt{W}\y,\xi)$. Next we introduce $F(\x,\xi)$ in such a way that the following relation holds:
$$\psi(y,\xi) = \max_{\x\in\R^{nm}}\left\{\la\y, \sqrt{W}\x\ra - F(\x,\xi)\right\}.$$ Note that for $\x(\sqrt{W}\y,\xi) \eqdef \argmax_{\x\in\R^{nm}}\left\{\la\y, \sqrt{W}\x\ra - F(\x,\xi)\right\}$ Demyanov--Danskin's theorem \cite{Rockafellar2015} states that $\nabla \psi(\y, \xi) = \sqrt{W}\x(\sqrt{W}\y,\xi)$ where the gradient is taken with respect the first variable. Finally, our definitions give us new relations: $\x(\sqrt{W}\y) = \EE_\xi[\x(\sqrt{W}\y,\xi)]$ and $\nabla \psi(\y) = \EE_\xi[\nabla \psi(\y,\xi)]$, where $\x(\y) \eqdef \argmax_{\x\in\R^{nm}}\left\{\la\y,\x\ra - F(\x)\right\} = \nabla\varphi(\y)$ and the last equality is again due to Demyanov-Danskin theorem.}

\ed{W}e suppose that  $\psi(\y)$ is
\ed{known only through the stochastic first-order} oracle $\nabla \psi (\y, \xi),$ satisfying the following assumption for all $\y\in\R^{nm}$\footnote{\pd{We believe that the light-tail assumption can be relaxed to a more general setting \cite{dvurechensky2018parallel}.}}:
\ed{\begin{eqnarray*}
    \E_\xi \exp \left( {\|\x(\y, \xi) - \x (\y)\|^2_2}/{\sigma_\x^2}\right) &\leq& \exp(1).
\end{eqnarray*}
Note that this implies
\begin{eqnarray*}
    \E_\xi \exp \left( {\|\nabla\psi(\y, \xi) - \nabla\psi (\y)\|^2_2}/{\sigma_\psi^2}\right) &\leq& \exp(1).
\end{eqnarray*}
for all $\y\in\R^{nm}$, where $\sigma_\psi^2 = \lambda_{\max}(W)\sigma_\x^2$.}

      \pd{We assume that the function $\psi$ is $L_\psi$-smooth. If, the primal objective is $\mu$-strongly convex, then $L_\psi \leq {\lm_{\max}(W)}/{\mu}$.}  
       Moreover, we assume \pd{that} we can construct an approximation for $\nabla \psi(\y)$ using batches of  size  $r$ in the following form:
        \begin{align}\label{eq:batched_estimates}
        \nabla^{r} \psi(\y, \{\xi_i\}^r_{i=1}) = \frac{1}{r}\sum_{i=1}^r \nabla \psi(\y,\xi_i)
        \end{align}
        \ed{and, similarly,
        $$
        \x(\sqrt{W}\y, \{\xi_i\}_{i=1}^r) = \frac{1}{r}\sum\limits_{i=1}^r\x(\sqrt{W}\y,\xi_i).
        $$}

  \begin{algorithm}[t]
\caption{\pd{Dual Stochastic  Algorithm}}
\label{Alg:DualStochAlg}          
 \begin{algorithmic}[1]
   \Require Starting point $\Blm^0 = \y^0=\Bzeta^0=\x^0= 0$, number of iterations $N$, $C_0=\alpha_0=0$,
                \For{$k=0,\dots, N-1$}
                \State
                \vspace{-0.35cm}
             \begin{equation}\label{eq:Alg_const}
                   \quad \pd{A_{k+1} = A_{k} + \alpha_{k+1}= 2L_\psi\alpha_{k+1}^2}
                   \end{equation}
                   \vspace{-0.7cm}
                \State
                \begin{equation}\label{eq:Alg_lambda}
                \Blm^{k+1} = (\alpha_{k+1}\Bzeta^k + A_k \y^k)/{A_{k+1}}.
                \end{equation}
                \vspace{-0.3cm}
                \dd{\State 
                Calculate $\nabla^{r _{k+1}} \psi(\Blm_{k+1},\{\xi_s\}_{s=1}^{r_{k+1}})$ according to \eqref{eq:batched_estimates} with batch size $$r_{k+1} = {O\left(\max \left\{ 1, { \sigma_\psi^2 {\alpha}_{k+1} \ln(N/\delta)}/{\e} \right\}\right)}$$}
                \vspace{-0.7cm}
                \State 
             \vspace{-0.3cm}
             \begin{eqnarray}\label{eq:Alg_zeta}
                \Bzeta^{k+1}= \Bzeta^{k} - \alpha_{k+1} \nabla^{r _{k+1}} \psi(\Blm_{k+1},\{\xi_s\}_{s=1}^{r_{k+1}}).
                \end{eqnarray}
                \vspace{-0.7cm}
                \State 
                \vspace{-0.3cm}
                \begin{eqnarray}\label{eq:Alg_y}
               \y^{k+1} =(\alpha_{k+1}\Bzeta^{k+1} + A_k \y^k)/{A_{k+1}}.
              \end{eqnarray}
                \EndFor
                \State{Set $\x^{N} = \frac{1}{A_{N}}\sum_{k=0}^{N} \alpha_k \x(\ed{\sqrt{W}}\Blm^k,\{\xi_i\}_{i=1}^{r_k}).$}
    \Ensure  ${\x}^{N}$, $\y^{N}$.    
\end{algorithmic}
\end{algorithm}

\begin{theorem}
\label{Th:stoch_err}
Assume that $F$ is $\mu$-strongly convex and  $\|\nabla F(\x^*)\|_2 = M_F$. Let
     $\varepsilon>0$ be a desired accuracy. Assume that at each iteration of Algorithm \ref{Alg:DualStochAlg} the approximation for $\nabla \psi(\y)$ is chosen according to \eqref{eq:batched_estimates} with batch size
     $r_k = \pd{\Omega\big(\max \big\{ 1, { \sigma^2_\psi {\alpha}_k \ln(N/\delta)}/{\e}\big\}\big) }$. \ed{Assume additionally that $F$ is $L_F$-Lipschitz continuous on the set $B_{R_F}(0) = \{\x\in \R^{nm}\mid \|\x\|_2 \le R_F\}$ where $R_F = {\Omega}\left(\max\left\{\frac{R_\y}{A_N}\sqrt{\frac{6C_2H}{\lambda_{\max}(W)}}, \frac{\lambda_{\max}(\sqrt{W})JR_\y}{\mu}, R_\x\right\}\right)$, $R_\y$ is such that $\|\y^*\|_2 \leq R_\y$, $\y^*$ being an optimal solution of the dual problem and $R_\x = \|\x(\sqrt{W}\y^*)\|_2$.}
Then, after  $N = \ed{\widetilde{O}}\big(\sqrt{({M^2_{F}}/{\mu\e})\chi(W)} \big)$ iterations, the outputs $\x^N$ and $\y^N$ of Algorithm \ref{Alg:DualStochAlg} satisfy 
\begin{align}
    F(\x^N) - F(\x^*) \leq \e, \quad \pd{\|\sqrt{W}\x^N\|_2 \leq {\e}/{R_\y}}
\end{align}
with probability at least $ 1-\ed{4}\delta$, \pd{where $\delta \in \left(\pd{0},{1}/{\ed{4}}\right)$}, \eg{$\ln(N/\delta) \ge 3$}.

Moreover, the number of stochastic oracle calls for the dual function $\nabla \vp_i$ per node $i=1,\dots m$ is
{\small
\begin{align*}
    O\left(\max\left\{  \frac{\sigma_{\psi}^2 M_F^2}{\e^2\lm_{\min}^+(W)} \ln\left(\frac{1}{\delta}\sqrt{\frac{M^2_F}{\mu\e}\chi(W)} \right), \sqrt {\frac{M^2_F}{\mu\e}\chi(W)} \right\} \right)
\end{align*}
}\\
\end{theorem}

To prove the theorem we first state a number of technical lemmas.

\begin{lemma}\label{lem:alpha_estimate}
     For the sequence $\alpha_{k+1}$ defined in \eqref{eq:Alg_const} we have for all $k\ge 0$
     \begin{equation}\label{eq:alpha_estimate}
         \alpha_{k+1} \le \widetilde{\alpha}_{k+1} \eqdef \frac{k+2}{2\pd{L_\psi}}.
     \end{equation}
\end{lemma}
 \begin{lemma}\label{lem:new_recurrence_lemma}
     Let $A, B,$ and \dd{$\{r_i\}_{i=0}^N$} be non-negative numbers such that \dd{for all $l=1,\ldots,N$ }
    \begin{equation}\label{eq:new_bound_for_r_l}
         \frac{1}{2}r_l^2 \le Ar_0^2 + B\frac{r_0}{N}\sqrt{\sum\limits_{k=0}^{l-1}(k+2)r_k^2}.
     \end{equation}
     Then $r_l \le Cr_0$,
     where $C$ is such positive number that $C^2 \ge \max\{1, 2A+2BC\}.$ 
 \end{lemma}
 The proof of the Lemma is followed from induction.\\

\begin{lemma}\label{lem:tails_estimate}
    Let the sequences of non-negative numbers $\{\alpha_k\}_{k\ge 0}$, random non-negative variables $\{R_k\}_{k\ge 0}$ and random vectors $\{\eta^k\}_{k\ge 0}$ and $\{a^k\}_{k\ge 0}$  for all $l = 1,\ldots,N$ satisfy
    \begin{eqnarray}
        \frac{1}{2}R_l^2 \le \eg{A} + \eg{u}\sum\limits_{k=0}^{l-1}\alpha_{k+1}\la\eta^{k+1}, a^k\ra+ \eg{c}\sum\limits_{k=0}^{l-1}\alpha_{k+1}^2\|\eta^{k+1}\|_2^2\label{eq:radius_recurrence}
    \end{eqnarray}
    where \eg{$A$ is deterministic non-negative number}, $\|a^k\|_2 \le \eg{d}\widetilde{R}_k$, \eg{$d\ge 1$ is some positive deterministic constant} and $\widetilde{R}_k = \max\{\widetilde{R}_{k-1}, R_k\}$ for all $k\ge 1$, $\widetilde{R}_0 = R_0$, $\widetilde{R}_k$ depends only on $\eta_0,\ldots,\eta^k$.\\
    Moreover, assume, vector $a^k$ is a function of $\eta^0,\ldots,\eta^{k-1}$  $\forall k\ge 1$, $a^0$ is a deterministic vector, and  $\forall  k \ge 0$, 
    \begin{eqnarray}
&        \EE\left[\eta^k\mid \{\eta^j\}_{j=0}^{k-1}\right] = 0, \notag \\
& \EE\left[\exp\left({\|\eta^k\|_2^2}{\sigma_k^{-2}}\right)\mid\{\eta^j\}_{j=0}^{k-1}\right] \le \exp(1), \label{eq:eta_k_properties}
    \end{eqnarray}
    \eg{$\alpha_{k+1} \le \widetilde{\alpha}_{k+1} = D(k+2)$, $\sigma_k^2 \le (C\varepsilon)/(\widetilde{\alpha}_{k+1}\ln(N/\delta))$ for some $D, C>0$, $\varepsilon > 0$}. \eg{If additionally $\varepsilon \le {HR_0^2}/{N^2}$,} then with probability at least $1-2\delta$ the inequalities
  \begin{eqnarray}
        &\widetilde{R}_l \le JR_0 \quad \text{ and }\label{eq:tails_estimate_radius}\\
        &\eg{u}\sum_{k=0}^{l-1}\alpha_{k+1}\la\eta^{k+1}, a^k\ra + \eg{c}\sum_{k=0}^{l-1}\alpha_{k+1}^2\|\eta^{k+1}\|_2^2 \notag \\
        &\hspace{1cm}\le \left(\eg{24cCDH + udC_1\sqrt{CDHJg(N)}}\right)R_0^2\label{eq:tails_estimate_stoch_part}
  \end{eqnarray}
    hold $\forall l=1,\ldots,N$ simultaneously. Here \eg{$C_1$ is some positive constant, $g(N) = \big(\ln\left({N}/{\delta}\right) + \ln\ln\left({B}/{b}\right)\big)/{\ln\left({N}/{\delta}\right)}$,
    \begin{eqnarray*}
    B = 2d^2CDHR_0^2\Big(2A + ud\widetilde{R}_0^2 &\\ &\hspace{-2cm}+12CD\varepsilon\left(2c+ud\right)N(N+3)\Big)(2ud)^N,
    \end{eqnarray*}
    $b = \sigma_0^2\widetilde{\alpha}_{1}^2d^2\widetilde{R}_0^2$ and
    \begin{eqnarray*}
    J = \max\Big\{1, udC_1\sqrt{CDH g(N)} &\\
    &\hspace{-3cm}+ \sqrt{u^2d^2C_1^2CDH g(N) + \frac{2A}{R_0^2} + 48cCDH}\Big\}.
    \end{eqnarray*}}
\end{lemma}

\subsection{Example: Computation of Wasserstein Barycenters}
It may seem that the problem with dual stochastic oracle is artificial. Next, we present the regularized Wasserstein barycenter problem \cite{cuturi2016smoothed,kroshnin2019complexity,Uribe2018,dvurechensky2018decentralize}, which is a recent example of a function with stochastic dual oracle,
\begin{align}\label{bary_problem}
\min_{p\in S_n(1)}\sum_{i=1}^m \W_{\mu, q_i}(p),
\end{align}
where 
$
    \W_{\mu, q_i}(p) = \min\limits_{\substack{\pi \mathbf{1} = p,  \pi^T \mathbf{1}  = q \\ \pi\geq 0} }\left\{ \la C, \pi \ra +\mu \la \pi \ln \pi \ra \right\}.
$\\
Here $C$ is a transportation cost matrix, $p,q$ are elements of standard probability simplex, logarithm of a matrix is taken componentwise.
Problem \eqref{bary_problem} is not easily tractable in the distributed setting since cost of approximating of the gradient of $\W_{\mu, q_i}(p)$ requires to solve a large-scale minimization problem. On the other hand, as it is shown in \cite{cuturi2016smoothed},
\begin{align*}
    &\W_{\mu, q_i}(p) = \max_{ u \in \R^n}\left\{ \la u, p \ra - \W_{q, \mu}^*(u) \right\} 
    \notag \\
    & \W_{q, \mu}^*(u) =  \mu\sum_{j=1}^nq_j \ln \left( \frac{1}{q_j}\sum_{i=1}^n\exp\left( \frac{-C_{ij}+ u_i}{\mu} \right) \right).
\end{align*}
So, the conjugate function has an explicit expression and its gradient can be calculated explicitly. Moreover, as the conjugate function has the form of finite-sum, we can use randomization and take a component $i$ with probability $q_i$.
As a corollary of our general Theorem \ref{Th:stoch_err}, we obtain 
\begin{corollary}
Taking  the batch size  $r_k = O\big(({\sigma^2_\psi \alpha_k \ln(N/\beta)}/{\e\mu})\big)$, where $\sigma^2_\psi = m\lm_{\max}(W)$  after $N = O\big(\sqrt{({M^2_F}/{\mu\e})\chi(W)} \big)$ iterations
the following holds for the output $\p^N$ of Algorithm~\ref{Alg:DualStochAlg} with probability at least $ 1-\ed{4}\delta$, where $\delta \in \left(\ed{0},{1}/{\ed{4}}\right)$ is such that $({1+\sqrt{\ln({1}/{\delta})}})/{\sqrt{\ln({N}/{\delta})}} \le 2$.
\begin{align*}
\sum_{i=1}^m \W_{\mu, q_i}(\p^N_i)- \sum_{i=1}^m  \W_{\mu, q_i}(p^*) \leq \e, ~ \|\sqrt{W}\p^N\|_2 \leq  \e/R_{\y}.
\end{align*}
Moreover, the total complexity per node is
{\small
\begin{align*}
    O\left( n\max\left\{ \frac{m M_F^2}{\e^2}\chi \ln\left(\frac{1}{\delta} \sqrt{\frac{M^2_F}{\mu\e}\chi} \right), ~ \sqrt {\frac{M^2_F}{\mu\e}\chi} \right\} \right),
\end{align*}
}
where  $M_F\dd{^2} = 2nm\|C\|^2_\infty$\cite{kroshnin2019complexity} and $\chi = \chi(W)$ . 

\end{corollary}


\section{Conclusion}\label{sec:conclusions}
We consider primal-dual distributed accelerated gradient method for stochastic finite-sum minimization. One of the key features of our analysis are large deviations bounds for the error of the algorithms. Moreover, we show that the proposed method has optimal communication complexity, up to logarithmic factors. For the proposed method we provide an explicit oracle and communication complexity analysis. We illustrate the dual approach by the Wasserstein barycenter problem. As a future work we consider extending these results for different classes of problems, i.e., non-smooth and/or also strongly convex problems. 

\textbf{Acknowledgements:} We are grateful to A. Nemirovski for fruitful discussions.

\bibliographystyle{IEEEtran} 
\bibliography{PD_references,Dvinskikh,all_refs3}

\newpage
\onecolumn
\section{Appendix}
\subsection{Auxiliary results}\label{sec:aux_results}
In this subsection, we present the results from other papers that we rely on in our proofs.
\begin{lemma}[Lemma~2 from \cite{jin2019short}]\label{lem:jin_lemma_2}
    For random vector $\xi \in \R^n$  following statements are equivalent up to absolute constant difference in $\sigma$.
    \begin{enumerate}
        \item Tails: $\PP\left\{\|\xi\|_2 \ge \gamma\right\} \le 2 \exp\left(-\frac{\gamma^2}{2\sigma^2}\right)$ $\forall \gamma \ge 0$.
        \item Moments: $\left(\EE\left[\xi^p\right]\right)^{\frac{1}{p}} \le \sigma\sqrt{p}$ for any positive integer $p$.
        \item Super-exponential moment: $\EE\left[\exp\left(\frac{\|\xi\|_2^2}{\sigma^2}\right)\right] \le \exp(1)$.
    \end{enumerate}
\end{lemma}

\begin{lemma}[Corollary~8 from \cite{jin2019short}]\label{lem:jin_corollary}
    Let $\{\xi_k\}_{k = 1}^N$ be a sequence of random vectors with values in $\R^n$ such that for $k=1,\ldots, N$ and for all $\gamma \ge 0$
    \begin{equation*}
        \EE\left[\xi_k\mid \xi_1,\ldots,\xi_{k-1}\right] = 0,\quad \EE\left[\|\xi_k\|_2 \ge \gamma \mid \xi_1,\ldots,\xi_{k-1}\right] \le \exp\left(-\frac{\gamma^2}{2\sigma_k^2}\right)\quad \text{almost surely,}
    \end{equation*}
    where $\sigma_k^2$ belongs to the filtration $\sigma(\xi_1,\ldots,\xi_{k-1})$ for all $k=1,\ldots, N$. Let $S_N = \sum\limits_{k=1}^N\xi_k$. Then there exists an absolute constant $C_1$ such that for any fixed $\delta > 0$ and $B > b > 0$ with probability at least $1 - \delta$:
    \begin{equation*}
        \text{either } \sum\limits_{k=1}^N\sigma_k^2 \ge B \quad \text{or} \quad \|S_N\|_2 \le C_1\sqrt{\max\left\{\sum\limits_{k=1}^N\sigma_k^2,b\right\}\left(\ln\frac{2n}{\delta} + \ln\ln\frac{B}{b}\right)}.
    \end{equation*}
\end{lemma}

\begin{lemma}[corollary of Theorem~2.1, item (ii) from \cite{juditsky2008large}]\label{lem:jud_nem_large_dev}
    Let $\{\xi_k\}_{k = 1}^N$ be a sequence of random vectors with values in $\R^n$ such that
    \begin{equation*}
        \EE\left[\xi_k\mid \xi_1,\ldots,\xi_{k-1}\right] = 0 \text{ almost surely,}\quad k=1,\ldots,N
    \end{equation*}
    and let $S_N = \sum\limits_{k=1}^N\xi_k$. Assume that the sequence $\{\xi_k\}_{k = 1}^N$ satisfy ``light-tail'' assumption:
    \begin{equation*}
        \EE\left[\exp\left(\frac{\|\xi_k\|_2^2}{\sigma_k^2}\right)\mid \xi_1, \ldots,\xi_{k-1}\right] \le \exp(1) \text{ almost surely,}\quad k = 1,\ldots,N,
    \end{equation*}
    where $\sigma_1,\ldots,\sigma_N$ are some positive numbers. Then for all $\gamma \ge 0$
    \begin{equation}
        \PP\left\{\|S_N\| \ge \left(\sqrt{2} + \sqrt{2}\gamma\right)\sqrt{\sum\limits_{k=1}^N\sigma_k^2}\right\} \le \exp\left(-\frac{\gamma^2}{3}\right).
    \end{equation}
\end{lemma}

\subsection{Proof of Theorem~\ref{thm:3}}\label{sec:proof_thm_3}
For Algorithm \ref{Alg:DualNFGM} the following holds
\begin{align*}
    F(\x^N)+\vp(\bar \y^N) \leq \frac{L_\psi R^2_{\bar\y}}{N^2},
    \end{align*}
\pd{where $R_{\bar\y}$ is such that $\|\bar \y^*\| \leq R_{\bar\y}$ is the radius of the solution.}
\pd{As it follows from \cite{lan2017communication},  $R_{\bar \y}$ can be taken as $R_{\bar \y}^2 = \nicefrac{\|\nabla F(\x^*)\|_2^2}{\lm^+_{\min}(W)}$.} 
Since the Lipschitz constant for the dual function $\psi$ is $L_\psi = \nicefrac{\lm_{\max}(W)}{\mu}$, we get the statement of the theorem.

\subsection{Proof of Theorem~\ref{Th:stoch_err}}
The proof includes several steps. We start with the proofs of the technical lemmas. For convenience we repeat statements of lemmas again.
\begin{lemma}\label{lem:alpha_estimate_appendix}
     For the sequence $\alpha_{k+1}$ defined in \eqref{eq:Alg_const} we have for all $k\ge 0$
     \begin{equation}\label{eq:alpha_estimate_appendix}
         \alpha_{k+1} \le \widetilde{\alpha}_{k+1} \eqdef \frac{k+2}{2\pd{L_\psi}}.
     \end{equation}
\end{lemma}
\begin{proof}
    We prove \eqref{eq:alpha_estimate_appendix} by induction. For $k=0$ equation \eqref{eq:Alg_const} gives us $\alpha_1 = 2L_\psi\alpha_1^2 \Longleftrightarrow \alpha_1 = \frac{1}{2L_\psi}$. Next we assume that \eqref{eq:alpha_estimate_appendix} holds for all $k\ge l-1$ and prove it for $k=l$:
    \begin{eqnarray*}
         2L_\psi\alpha_{l+1}^2 &\overset{\eqref{eq:Alg_const}}{=}& \sum\limits_{i=1}^{l+1}\alpha_i \overset{\eqref{eq:alpha_estimate}}{\le} \alpha_{l+1} + \frac{1}{2L_\psi}\sum\limits_{i=1}^{l}(i+1) = \alpha_{l+1} + \frac{l(l+3)}{4L_\psi}.
    \end{eqnarray*}
    This quadratic inequality implies that $\alpha_{k+1} \le \frac{1+\sqrt{4k^2 + 12k + 1}}{4L_\psi} \le \frac{1+\sqrt{(2k+3)^2}}{4L_\psi} \le \frac{2k+4}{4L_\psi} = \frac{k+2}{2L_\psi}$.
\end{proof}

\begin{lemma}\label{lem:new_recurrence_lemma_appendix}
     Let $A, B,$ and \dd{$\{r_i\}_{i=0}^N$} be non-negative numbers such that \dd{for all $l=1,\ldots,N$ }
    \begin{equation}\label{eq:new_bound_for_r_l_appendix}
         \frac{1}{2}r_l^2 \le Ar_0^2 + B\frac{r_0}{N}\sqrt{\sum\limits_{k=0}^{l-1}(k+2)r_k^2}.
     \end{equation}
     Then
     \begin{equation}\label{eq:new_recurrence_lemma_appendix}
         r_l \le Cr_0,
     \end{equation}
     where $C$ is such positive number that $C^2 \ge \max\{1, 2A+2BC\}$, i.e.\ one can choose $C = \max\{1, B + \sqrt{B^2 + 2A}\}$.
\end{lemma}
\begin{proof}
     We prove \eqref{eq:new_recurrence_lemma_appendix} by induction. For $l=0$ the inequality $r_l \le Cr_0$ trivially follows since $C \ge 1$. Next we assume that \eqref{eq:new_recurrence_lemma_appendix} holds for some $l < N$ and prove it for $l+1$:
     \begin{eqnarray*}
        r_{l+1} &\overset{\eqref{eq:new_bound_for_r_l_appendix}}{\le}& \sqrt{2}\sqrt{Ar_0^2 + B \frac{r_0}{N}\sqrt{\sum\limits_{k=0}^{l}(k+2)r_k^2}}
        \overset{\eqref{eq:new_recurrence_lemma_appendix}}{\le} r_0\sqrt{2}\sqrt{A + \frac{BC}{N}\sqrt{\sum\limits_{k=0}^l(k+2)}}\\
        &=& r_0\sqrt{2}\sqrt{A + \frac{BC}{N}\sqrt{\frac{(l+1)(l+2)}{2}}} \le r_0\sqrt{2}\sqrt{A + \frac{BC}{N}\sqrt{\frac{N(N+1)}{2}}}\\
        &\le& r_0\underbrace{\sqrt{2A + 2BC}}_{\le C} \le Cr_0.
     \end{eqnarray*}
 \end{proof}

\begin{lemma}\label{lem:tails_estimate_appendix}
    Let the sequences of non-negative numbers $\{\alpha_k\}_{k\ge 0}$, random non-negative variables $\{R_k\}_{k\ge 0}$ and random vectors $\{\eta^k\}_{k\ge 0}$ and $\{a^k\}_{k\ge 0}$  for all $l = 1,\ldots,N$ satisfy
    \begin{eqnarray}
        \frac{1}{2}R_l^2 \le \eg{A} + \eg{u}\sum\limits_{k=0}^{l-1}\alpha_{k+1}\la\eta^{k+1}, a^k\ra+ \eg{c}\sum\limits_{k=0}^{l-1}\alpha_{k+1}^2\|\eta^{k+1}\|_2^2\label{eq:radius_recurrence_appendix}
    \end{eqnarray}
    where \eg{$A$ is deterministic non-negative number}, $\|a^k\|_2 \le \eg{d}\widetilde{R}_k$, \eg{$d\ge 1$ is some positive deterministic constant} and $\widetilde{R}_k = \max\{\widetilde{R}_{k-1}, R_k\}$ for all $k\ge 1$, $\widetilde{R}_0 = R_0$, $\widetilde{R}_k$ depends only on $\eta_0,\ldots,\eta^k$. Moreover, assume, vector $a^k$ is a function of $\eta^0,\ldots,\eta^{k-1}$  $\forall k\ge 1$, $a^0$ is a deterministic vector, and  $\forall  k \ge 0$, 
    \begin{eqnarray}
\EE\left[\eta^k\mid \{\eta^j\}_{j=0}^{k-1}\right] = 0,\quad \EE\left[\exp\left({\|\eta^k\|_2^2}{\sigma_k^{-2}}\right)\mid\{\eta^j\}_{j=0}^{k-1}\right] \le \exp(1), \label{eq:eta_k_properties_appendix}
    \end{eqnarray}
    \eg{$\alpha_{k+1} \le \widetilde{\alpha}_{k+1} = D(k+2)$, $\sigma_k^2 \le \frac{C\varepsilon}{\widetilde{\alpha}_{k+1}\ln(\nicefrac{N}{\delta})}$ for some $D, C>0$, $\varepsilon > 0$}. \eg{If additionally $\varepsilon \le \nicefrac{HR_0^2}{N^2}$,} then with probability at least $1-2\delta$ the inequalities
  \begin{eqnarray}
        &\widetilde{R}_l \le JR_0 \quad \text{ and }\label{eq:tails_estimate_radius_appendix}\\
        &\eg{u}\sum_{k=0}^{l-1}\alpha_{k+1}\la\eta^{k+1}, a^k\ra + \eg{c}\sum_{k=0}^{l-1}\alpha_{k+1}^2\|\eta^{k+1}\|_2^2 \le \left(\eg{24cCDH + udC_1\sqrt{CDHJg(N)}}\right)R_0^2\label{eq:tails_estimate_stoch_part_appendix}
  \end{eqnarray}
    hold $\forall l=1,\ldots,N$ simultaneously. Here \eg{$C_1$ is some positive constant, $g(N) = \frac{\ln\left(\nicefrac{N}{\delta}\right) + \ln\ln\left(\nicefrac{B}{b}\right)}{\ln\left(\nicefrac{N}{\delta}\right)}$,
    \begin{eqnarray*}
    B = 2d^2CDHR_0^2\Big(2A + ud\widetilde{R}_0^2+12CD\varepsilon\left(2c+ud\right)N(N+3)\Big)(2ud)^N,
    \end{eqnarray*}
    $b = \sigma_0^2\widetilde{\alpha}_{1}^2d^2\widetilde{R}_0^2$ and
    \begin{eqnarray*}
    J = \max\left\{1, udC_1\sqrt{CDH g(N)} + \sqrt{u^2d^2C_1^2CDH g(N) + \frac{2A}{R_0^2} + 48cCDH}\right\}.
    \end{eqnarray*}}
\end{lemma}
\begin{proof}
     We start with applying Cauchy-Schwartz inequality to the second term in the right-hand side of \eqref{eq:radius_recurrence}:
    \begin{eqnarray}\label{eq:radius_recurrence2}
        \frac{1}{2}R_l^2 &\le& A + ud\sum\limits_{k=0}^{l-1}\alpha_{k+1}\|\eta^k\|_2\widetilde{R}_k + c\sum\limits_{k=0}^{l-1}\alpha_{k+1}^2\|\eta^k\|_2^2,\notag\\
        &\le& A + \frac{ud}{2}\sum\limits_{k=0}^{l-1}\widetilde{R}_k^2 + \left(c + \frac{ud}{2}\right)\sum\limits_{k=0}^{l-1}\widetilde{\alpha}_{k+1}^2\|\eta^k\|_2^2.
    \end{eqnarray}

    The idea of the proof is as following: estimate $R_N^2$ roughly, then apply Lemma~\ref{lem:jin_corollary} in order to estimate second term in the last row of \eqref{eq:radius_recurrence_appendix} and after that use the obtained recurrence to estimate right-hand side of \eqref{eq:radius_recurrence_appendix}.
    
    Using Lemma~\ref{lem:jud_nem_large_dev} we get that with probability at least $1 - \frac{\delta}{N}$
    \begin{eqnarray}
        \|\eta^k\|_2 &\le& \sqrt{2}\left(1 + \sqrt{3\ln\frac{N}{\delta}}\right)\sigma_k \le \sqrt{2}\left(1 + \sqrt{3\ln\frac{N}{\delta}}\right)\frac{\sqrt{C\varepsilon}}{\sqrt{\widetilde{\alpha}_{k+1}\ln\left(\frac{N}{\delta}\right)}}\notag\\
        &=& \left(\frac{1}{\sqrt{\widetilde{\alpha}_{k+1}\ln\left(\frac{N}{\delta}\right)}} + \sqrt{\frac{3}{\widetilde{\alpha}_{k+1}}}\right)\sqrt{2C\varepsilon} \le 2\sqrt{\frac{3}{\widetilde{\alpha}_{k+1}}}\sqrt{2C\varepsilon},\label{eq:eta_k_norm_bound}
    \end{eqnarray}
    where in the last inequality we use $\ln\frac{N}{\delta} \ge 3$.
    Using union bound we get that with probability $\ge 1 - \delta$ the inequality
    \begin{eqnarray*}
        \frac{1}{2}R_l^2 &\le& A + \frac{ud}{2}\sum\limits_{k=0}^{l-1}\widetilde{R}_k^2 + 24C\varepsilon\left(c+\frac{ud}{2}\right)\sum\limits_{k=0}^{l-1}\widetilde{\alpha}_{k+1}\\
        &\le& A + \frac{ud}{2}\sum\limits_{k=0}^{l-1}\widetilde{R}_k^2 + 24CD\varepsilon\left(c+\frac{ud}{2}\right)\sum\limits_{k=0}^{l-1}(k+2)\\
        &\le& A + \frac{ud}{2}\sum\limits_{k=0}^{l-1}\widetilde{R}_k^2 + 12CD\varepsilon\left(c+\frac{ud}{2}\right)l(l+3)
    \end{eqnarray*}
    holds for all $l=1,\ldots,N$ simultaneously. Note that the last row in the previous inequality is non-decreasing function of $l$. If we define $\hat{l}$ as the largest integer such that $\hat{l}\le l$ and $\widetilde{R}_{\hat{l}} = R_{\hat{l}}$, we will get that $R_{\hat{l}} = \widetilde{R}_{\hat{l}} = \widetilde{R}_{\hat{l}+1} = \ldots = \widetilde{R}_{l}$ and, as a consequence, with probability $\ge 1 - \delta$
    \begin{eqnarray*}
        \frac{1}{2}\widetilde{R}_l^2 &\le& A + \frac{ud}{2}\sum\limits_{k=0}^{\hat{l}-1}\widetilde{R}_k^2 + 12CD\varepsilon\left(c+\frac{ud}{2}\right)\hat{l}(\hat{l}+3)\\
        &\le& A + \frac{ud}{2}\sum\limits_{k=0}^{l-1}\widetilde{R}_k^2 + 12CD\varepsilon\left(c+\frac{ud}{2}\right)l(l+3),\quad \forall l=1,\ldots,N.
    \end{eqnarray*}
    Therefore, we have that with probability $\ge 1 - \delta$
    \begin{eqnarray}
        \widetilde{R}_l^2 &\le& 2A + ud\sum\limits_{k=0}^{l-1}\widetilde{R}_k^2 + 12CD\varepsilon\left(2c+ud\right)l(l+3)\notag\\
        &\le& 2A\underbrace{(1+ud)}_{\le 2ud} + \underbrace{(ud + u^2d^2)}_{\le 2u^2d^2}\sum\limits_{k=0}^{l-2}\widetilde{R}_k^2 + 12CD\varepsilon(2c+ud)\underbrace{(l(l+3) + ud(l-1)(l+2))}_{\le 2udl(l+3)}\notag\\
        &\le& 2ud\left(2A + ud\sum\limits_{k=0}^{l-2}\widetilde{R}_k^2 + 12CD\varepsilon\left(2c+ud\right)l(l+3)\right),\quad \forall l = 1,\ldots, N.\notag
    \end{eqnarray}
    Unrolling the recurrence we get that with probability $\ge 1 - \delta$
    \begin{eqnarray*}
      \widetilde{R}_l^2 &\le& \left(2A + ud\widetilde{R}_0^2 + 12CD\varepsilon\left(2c+ud\right)l(l+3)\right)(2ud)^l,\quad \forall l = 1,\ldots, N.
    \end{eqnarray*}
    We emphasize that it is very rough estimate, but we show next that such a bound does not spoil the final result too much.
    It implies that with probability $\ge 1-\delta$
    \begin{equation}\label{eq:bound_sum_squared_radius}
        \sum\limits_{k=0}^{l-1}\widetilde{R}_k^2 \le  l\left(2A + ud\widetilde{R}_0^2 + 12CD\varepsilon\left(2c+ud\right)l(l+3)\right)(2ud)^l,\quad \forall l=1,\ldots,N.
    \end{equation}
    
    Next we apply delicate result from \cite{jin2019short} which is presented in Section~\ref{sec:aux_results} as Lemma~\ref{lem:jin_corollary}. We consider random variables $\xi^k = \widetilde{\alpha}_{k+1}\la\eta^k, a^k \ra$. Note that $\EE\left[\xi^k\mid \xi^0,\ldots,\xi^{k-1}\right] = \widetilde{\alpha}_{k+1}\left\la\EE\left[\eta^k\mid \eta^0,\ldots,\eta^{k-1}\right], a^k \right\ra = 0$ and
    \begin{eqnarray*}
        \EE\left[\exp\left(\frac{(\xi^k)^2}{\sigma_k^2\widetilde{\alpha}_{k+1}^2d^2\widetilde{R}_k^2}\right)\mid \xi^0,\ldots,\xi^{k-1}\right] &\le& \EE\left[\exp\left(\frac{\widetilde{\alpha}_{k+1}^2\|\eta^k\|_2^2 d^2\widetilde{R}_k^2}{\sigma_k^2\widetilde{\alpha}_{k+1}^2d^2\widetilde{R}_k^2}\right)\mid \eta^0,\ldots,\eta^{k-1}\right]\\
        &=& \EE\left[\exp\left(\frac{\|\eta^k\|_2^2}{\sigma_k^2}\right)\mid \eta^0,\ldots,\eta^{k-1}\right] \le \exp(1)
    \end{eqnarray*}
    due to Cauchy-Schwartz inequality and assumptions of the lemma. If we denote $\hat\sigma_k^2 = \sigma_k^2\widetilde{\alpha}_{k+1}^2d^2\widetilde{R}_k^2$ and apply Lemma~\ref{lem:jin_corollary} with $B = 2d^2CDHR_0^2\left(2A + ud\widetilde{R}_0^2 + 12CD\varepsilon\left(2c+ud\right)N(N+3)\right)(2ud)^N$ and $b=\hat{\sigma}_0^2$, we get that for all $l=1,\ldots,N$ with probability $\ge 1-\frac{\delta}{N}$
    \begin{equation*}
        \text{either} \sum\limits_{k=0}^{l-1}\hat\sigma_k^2 \ge B \text{ or } \left|\sum\limits_{k=0}^{l-1}\xi^k\right| \le C_1\sqrt{\sum\limits_{k=0}^{l-1}\hat\sigma_k^2\left(\ln\left(\frac{N}{\delta}\right) + \ln\ln\left(\frac{B}{b}\right)\right)}
    \end{equation*}
    with some constant $C_1 > 0$ which does not depend on $B$ or $b$. Using union bound we obtain that with probability $\ge 1 - \delta$
    \begin{equation*}
        \text{either} \sum\limits_{k=0}^{l-1}\hat\sigma_k^2 \ge B \text{ or } \left|\sum\limits_{k=0}^{l-1}\xi^k\right| \le C_1\sqrt{\sum\limits_{k=0}^{l-1}\hat\sigma_k^2\left(\ln\left(\frac{N}{\delta}\right) + \ln\ln\left(\frac{B}{b}\right)\right)}
    \end{equation*}
    and it holds for all $l=1,\ldots, N$ simultaneously. Note that with probability at least $1-\delta$
    \begin{eqnarray*}
        \sum\limits_{k=0}^{l-1}\hat\sigma_k^2 &=& d^2\sum\limits_{k=0}^{l-1}\sigma_k^2\widetilde{\alpha}_{k+1}^2\widetilde{R}_k^2 \le d^2\sum\limits_{k=0}^{l-1}\frac{C\varepsilon}{\ln\frac{N}{\delta}}\widetilde{\alpha}_{k+1}\widetilde{R}_k^2\\
        &\le& \frac{d^2CDHR_0^2}{N^2\ln\frac{N}{\delta}}\sum\limits_{k=0}^{l-1}(k+2)\widetilde{R}_k^2 \le \frac{d^2CDHR_0^2}{3N}\cdot\frac{N+1}{N}\sum\limits_{k=0}^{l-1}\widetilde{R}_k^2\\
        &\overset{\eqref{eq:bound_sum_squared_radius}}{\le}& \frac{d^2CDHR_0^2}{N}l\left(2A + ud\widetilde{R}_0^2 + 12CD\varepsilon\left(2c+ud\right)l(l+3)\right)(2ud)^l\\
        &\le& \frac{B}{2}
    \end{eqnarray*}
    for all $l=1,\ldots, N$ simultaneously. Using union bound again we get that with probability $\ge 1 - 2\delta$ the inequality
    \begin{equation}\label{eq:bound_inner_product}
        \left|\sum\limits_{k=0}^{l-1}\xi^k\right| \le C_1\sqrt{\sum\limits_{k=0}^{l-1}\hat\sigma_k^2\left(\ln\left(\frac{N}{\delta}\right) + \ln\ln\left(\frac{B}{b}\right)\right)}
    \end{equation}
    holds for all $l=1,\ldots,N$ simultaneously.
    
    Note that we also proved that \eqref{eq:eta_k_norm_bound} is in the same event together with \eqref{eq:bound_inner_product} and holds with probability $\ge 1 - 2\delta$. Putting all together in \eqref{eq:radius_recurrence_appendix}, we get that with probability at least $1-2\delta$ the inequality
    \begin{eqnarray*}
        \frac{1}{2}\widetilde{R}_l^2 &\overset{\eqref{eq:radius_recurrence}}{\le}&  A + u\sum\limits_{k=0}^{l-1}\alpha_{k+1}\la\eta^k, a^k\ra + c\sum\limits_{k=0}^{l-1}\alpha_{k+1}^2\|\eta^k\|_2^2\\
        &\overset{\eqref{eq:bound_inner_product}}{\le}& A + uC_1\sqrt{\sum\limits_{k=0}^{l-1}\hat\sigma_k^2\left(\ln\left(\frac{N}{\delta}\right) + \ln\ln\left(\frac{B}{b}\right)\right)} + 24cC\varepsilon\sum\limits_{k=0}^{l-1}\widetilde{\alpha}_{k+1}
    \end{eqnarray*}
    holds for all $l=1,\ldots,N$ simultaneously. For brevity, we introduce new notation: $g(N) = \frac{\ln\left(\frac{N}{\delta}\right) + \ln\ln\left(\frac{B}{b}\right)}{\ln\left(\frac{N}{\delta}\right)} \approx 1$ (neglecting constant factor). Using our assumption $\sigma_k^2 \le \frac{C\varepsilon}{\widetilde{\alpha}_{k+1}\ln\left(\frac{N}{\delta}\right)}$ and definition $\hat\sigma_k^2 = \sigma_k^2\widetilde{\alpha}_{k+1}^2d^2\widetilde{R}_k^2$ we obtain that with probability at least $1-2\delta$ the inequality
    \begin{eqnarray}\label{eq:radius_recurrence_large_prob}
        \frac{1}{2}\widetilde{R}_l^2 &\le& A + u\sum\limits_{k=0}^{l-1}\alpha_{k+1}\la\eta^k, a^k\ra + c\sum\limits_{k=0}^{l-1}\alpha_{k+1}^2\|\eta^k\|_2^2 \notag\\
        &\le& A + 24cC\varepsilon\sum\limits_{k=0}^{l-1}\widetilde{\alpha}_{k+1} + udC_1\sqrt{C\varepsilon g(N)}\sqrt{\sum\limits_{k=0}^{l-1}\widetilde{\alpha}_{k+1}\widetilde{R}_k^2}\notag\\
        &\le& A + 24cCD\varepsilon\sum\limits_{k=0}^{l-1}(k+2) + udC_1\sqrt{CD\varepsilon g(N)}\sqrt{\sum\limits_{k=0}^{l-1}(k+2)\widetilde{R}_k^2}\notag\\
        &\le& A + 24cCD\frac{HR_0^2}{N^2}\frac{l(l+1)}{2} + udC_1\sqrt{CD\frac{HR_0^2}{N^2} g(N)}\sqrt{\sum\limits_{k=0}^{l-1}(k+2)\widetilde{R}_k^2}\notag\\
        &\le& \left(\frac{A}{R_0^2} + 24cCDH\right)R_0^2 + \frac{udC_1R_0}{N}\sqrt{CDH g(N)}\sqrt{\sum\limits_{k=0}^{l-1}(k+2)\widetilde{R}_k^2}
    \end{eqnarray}
    holds for all $l=1,\ldots,N$ simultaneously. Next we apply Lemma~\ref{lem:new_recurrence_lemma} with $A = \frac{A}{R_0^2} + 24cCDH$, $B = udC_1\sqrt{CDH g(N)}$, $r_k = \widetilde{R}_k$ and get that with probability at least $1-2\delta$ inequality
    \begin{eqnarray*}
        \widetilde{R}_l \le JR_0
    \end{eqnarray*}
    holds for all $l=1,\ldots,N$ simultaneously with $$J = \max\left\{1, udC_1\sqrt{CDH g(N)} + \sqrt{u^2d^2C_1^2CDH g(N) + \frac{2A}{R_0^2} + 48cCDH}\right\}.$$ It implies that with probability at least $1-2\delta$ the inequality
    \begin{eqnarray*}
        A + u\sum\limits_{k=0}^{l-1}\alpha_{k+1}\la\eta^k, a^k\ra + c\sum\limits_{k=0}^{l-1}\alpha_{k+1}^2\|\eta^k\|_2^2 &\\
        &\hspace{-2cm}\le\left(\frac{A}{R_0^2} + 24cCDH\right)R_0^2 + \frac{udC_1R_0^2}{N}\sqrt{CDH g(N)}\sqrt{\sum\limits_{k=0}^{l-1}(k+2)J}\\
        &\hspace{-2.8cm}\le A + \left(24cCDH + udC_1\sqrt{CDHJg(N)}\frac{1}{N}\sqrt{\frac{l(l+1)}{2}}\right)R_0^2\\
        &\hspace{-4.5cm}\le A + \left(24cCDH + udC_1\sqrt{CDHJg(N)}\right)R_0^2
    \end{eqnarray*}
    holds for all $l=1,\ldots,N$ simultaneously.
\end{proof}

\begin{lemma}[see also Theorem~1 from \cite{dvurechenskii2018decentralize}]\label{lem:psi_y_N_bound}
    For each iteration of Algorithm~\ref{Alg:DualStochAlg} we have
    \begin{eqnarray}
       A_N\psi(\y^N)  &\leq& \frac{1}{2}\|\Blm - \Bzeta^0\|_2^2 - \frac{1}{2}\|\Blm - \Bzeta^N\|_2^2 + \sum_{k=0}^{N-1} \alpha_{k+1} \left( \psi(\Blm^{k+1}) + \la \nabla \Psi(\Blm^{k+1}, \Bxi^{k+1}), \Blm - \Blm^{k+1}\ra \right)  \notag \\ 
 &&\quad +  \sum_{k=0}^{N-1}A_{k}\la \nabla \Psi(\Blm^{k+1}, \Bxi^{k+1})  - \nabla \psi(\Blm^{k+1}) , \y^k- \Blm^{k+1} \ra \notag\\
 &&\quad +  \sum_{k=0}^{N-1}\frac{A_{k+1}}{2L_\psi}\|\nabla \psi(\Blm^{k+1}) -\nabla \Psi(\Blm^{k+1}, \Bxi^{k+1})\|_{2}^2,\label{eq:psi_y_N_bound}
 \end{eqnarray}
 where we use the following notation for the stochastic approximation of $\nabla \psi(\Blm)$ according to \eqref{eq:batched_estimates}  
\begin{align}
    \nabla \Psi(\Blm^k, \Bxi^k):=\nabla^{r_{k}} \psi(\Blm^{k}, \{\xi^{k}_i\}_{i=1}^{r_{k}}),
\end{align}
 where $\Bxi^k = (\xi^k_1,\dots, \xi^k_{r_k})$.
\end{lemma}
\begin{proof}
     The proof of this lemma follows a similar way as in the proof of Theorem~1 from \cite{dvurechenskii2018decentralize}. We can rewrite the update rule for $\Bzeta^k$ in the equivalent way:
\begin{equation*}
    \Bzeta^k = \argmin\limits_{\Blm\in\R^n}\left\{\alpha_{k+1}\la\nabla\Psi(\Blm^{k+1},\Bxi^{k+1}), \Blm - \Blm^{k+1}\ra + \frac{1}{2}\|\Blm-\Bzeta^k\|_2^2\right\}.
\end{equation*}
From the optimality condition we have that for all $z\in\R^n$
\begin{equation}
    \la \Bzeta^{k+1} - \Bzeta^k + \alpha_{k+1}\nabla\Psi(\Blm^{k+1},\Bxi^{k+1}), \Blm - \Bzeta^{k+1} \ra \ge 0.\label{eq:z_k_optimality_cond}
\end{equation}
Using this we get
\begin{eqnarray*}
    \alpha_{k+1}\la\nabla\Psi(\Blm^{k+1},\Bxi^{k+1}), \Bzeta^k - \Blm \ra&\\
    &\hspace{-1cm}= \alpha_{k+1}\la\nabla\Psi(\Blm^{k+1},\Bxi^{k+1}), \Bzeta^k - \Bzeta^{k+1} \ra + \alpha_{k+1}\la\nabla\Psi(\Blm^{k+1},\Bxi^{k+1}), \Bzeta^{k+1} - \Blm \ra\\
    &\hspace{-2.5cm}\overset{\eqref{eq:z_k_optimality_cond}}{\le} \alpha_{k+1}\la\nabla\Psi(\Blm^{k+1},\Bxi^{k+1}), \Bzeta^k - \Bzeta^{k+1} \ra + \la \Bzeta^{k+1}-\Bzeta^k, \Blm - \Bzeta^{k+1} \ra.
\end{eqnarray*}
One can check via direct calculations that
\begin{equation*}
    \la a, b \ra \le \frac{1}{2}\|a+b\|_2^2 - \frac{1}{2}\|a\|_2^2 - \frac{1}{2}\|b\|_2^2, \quad \forall\; a,b\in\R^n.
\end{equation*}
Combining previous two inequalities we obtain
\begin{eqnarray*}
    \alpha_{k+1}\la\nabla\Psi(\Blm^{k+1},\Bxi^{k+1}), \Bzeta^k - \Blm \ra &\le& \alpha_{k+1}\la\nabla\Psi(\Blm^{k+1},\Bxi^{k+1}), \Bzeta^k - \Bzeta^{k+1} \ra - \frac{1}{2}\|\Bzeta^{k} - \Bzeta^{k+1}\|_2^2\\
    &&\quad+ \frac{1}{2}\|\Bzeta^k - \Blm\|_2^2 - \frac{1}{2}\|\Bzeta^{k+1} - \Blm\|_2^2.
\end{eqnarray*}
By definition of $\y^{k+1}$ and $\Blm^{k+1}$
\begin{equation*}
    \y^{k+1} = \frac{A_k \y^k + \alpha_{k+1}\Bzeta^{k+1}}{A_{k+1}} = \frac{A_k \y^k + \alpha_{k+1}\Bzeta^{k}}{A_{k+1}} + \frac{\alpha_{k+1}}{A_{k+1}}\left(\Bzeta^{k+1}-\Bzeta^k\right) = \Blm^{k+1} + \frac{\alpha_{k+1}}{A_{k+1}}\left(\Bzeta^{k+1}-\Bzeta^k\right).
\end{equation*}
Together with previous inequality, it implies
\begin{eqnarray*}
    \alpha_{k+1}\la\nabla\Psi(\Blm^{k+1},\Bxi^{k+1}), \Bzeta^k - \Blm \ra &\le& A_{k+1}\la\nabla\Psi(\Blm^{k+1},\Bxi^{k+1}), \Blm^{k+1} - \y^{k+1} \ra - \frac{A_{k+1}^2}{2\alpha_{k+1}^2}\|\Blm^{k+1} - \y^{k+1}\|_2^2\\
    &&\quad+ \frac{1}{2}\|\Bzeta^k - \Blm\|_2^2 - \frac{1}{2}\|\Bzeta^{k+1}-\Blm\|_2^2\\
    &\le& A_{k+1}\left(\la\nabla\Psi(\Blm^{k+1},\Bxi^{k+1}), \Blm^{k+1} - \y^{k+1} \ra - \frac{2L_\psi}{2}\|\Blm^{k+1} - \y^{k+1}\|_2^2\right)\\
    &&\quad+ \frac{1}{2}\|\Bzeta^k - \Blm\|_2^2 - \frac{1}{2}\|\Bzeta^{k+1}-\Blm\|_2^2\\
    &=& A_{k+1}\left(\la\nabla\psi(\Blm^{k+1}), \Blm^{k+1} - \y^{k+1} \ra - \frac{2L_\psi}{2}\|\Blm^{k+1} - \y^{k+1}\|_2^2\right)\\
    &&\quad + A_{k+1}\la\nabla\Psi(\Blm^{k+1},\Bxi^{k+1}) - \nabla\psi(\Blm^{k+1}), \Blm^{k+1} - \y^{k+1}\ra\\
    &&\quad+ \frac{1}{2}\|\Bzeta^k - \Blm\|_2^2 - \frac{1}{2}\|\Bzeta^{k+1}-\Blm\|_2^2.
\end{eqnarray*}
From Fenchel-Young inequality $\la a, b\ra \le \frac{1}{2\eta}\|a\|_2^2 + \frac{\eta}{2}\|b\|_2^2$, $a,b\in\R^n$, $\eta > 0$, we have
\begin{eqnarray*}
    \la\nabla\Psi(\Blm^{k+1},\Bxi^{k+1}) - \nabla\psi(\Blm^{k+1}), \Blm^{k+1} - \y^{k+1}\ra &\\
    &\hspace{-1.5cm}\le \frac{1}{2L_\psi}\left\|\nabla\Psi(\Blm^{k+1},\Bxi^{k+1}) - \nabla\psi(\Blm^{k+1})\right\|_2^2 + \frac{L_\psi}{2}\|\Blm^{k+1} - \y^{k+1}\|_2^2.
\end{eqnarray*}
Using this, we get
\begin{eqnarray}
    \alpha_{k+1}\la\nabla\Psi(\Blm^{k+1},\Bxi^{k+1}), \Bzeta^k - \Blm \ra &\le& A_{k+1}\left(\la\nabla\psi(\Blm^{k+1}), \Blm^{k+1} - \y^{k+1} \ra - \frac{L_\psi}{2}\|\Blm^{k+1} - \y^{k+1}\|_2^2\right)\notag\\
    &&\quad + \frac{A_{k+1}}{2L_\psi}\left\|\nabla\Psi(\Blm^{k+1},\Bxi^{k+1}) - \nabla\psi(\Blm^{k+1})\right\|_2^2\notag\\
    &&\quad+ \frac{1}{2}\|\Bzeta^k - \Blm\|_2^2 - \frac{1}{2}\|\Bzeta^{k+1}-\Blm\|_2^2\notag\\
    &\le& A_{k+1}\left(\psi(\Blm^{k+1}) - \psi(\y^{k+1})\right) + \frac{1}{2}\|\Bzeta^k - \Blm\|_2^2 - \frac{1}{2}\|\Bzeta^{k+1}-\Blm\|_2^2\notag\\
    &&\quad + \frac{A_{k+1}}{2L_\psi}\left\|\nabla\Psi(\Blm^{k+1},\Bxi^{k+1}) - \nabla\psi(\Blm^{k+1})\right\|_2^2,\label{eq:inner_prod_bound_1}
\end{eqnarray}
where the last inequality follows from the $L_\psi$-smoothness of $\psi(y)$. From the convexity of $\psi(y)$, we have
\begin{eqnarray}
    \la\nabla\Psi(\Blm^{k+1},\Bxi^{k+1}), \y^{k} - \Blm^{k+1} \ra &\notag\\
    &\hspace{-2cm} = \la\nabla\psi(\Blm^{k+1}), \y^{k} - \Blm^{k+1}\ra + \la\nabla\Psi(\Blm^{k+1},\Bxi^{k+1}) - \nabla\psi(\Blm^{k+1}), \y^{k} - \Blm^{k+1} \ra\notag\\
    &\hspace{-2.8cm}\le  \psi(\y^k) - \psi(\Blm^{k+1}) + \la\nabla\Psi(\Blm^{k+1},\Bxi^{k+1}) - \nabla\psi(\Blm^{k+1}),\y^{k} - \Blm^{k+1} \ra.\label{eq:inner_prod_bound_2}
\end{eqnarray}
By definition of $\Blm^{k+1}$ we have
\begin{equation}
    \alpha_{k+1}\left(\Blm^{k+1} - \Bzeta^k\right) = A_k\left(\y^k - \Blm^{k+1}\right)\label{eq:tilde_y^k+1-z^k_relation}.
\end{equation}
Putting all together, we get
\begin{eqnarray*}
    \alpha_{k+1}\la\nabla \Psi(\Blm^{k+1},\Bxi^{k+1}),\Blm^{k+1} - \Blm \ra &\\
    &\hspace{-3cm}= \alpha_{k+1}\la\nabla \Psi(\Blm^{k+1},\Bxi^{k+1}),\Blm^{k+1} - \Bzeta^k \ra + \alpha_{k+1}\la\nabla \Psi(\Blm^{k+1},\Bxi^{k+1}),\Bzeta^k - \Blm \ra\\
    &\hspace{-3.3cm}\overset{\eqref{eq:tilde_y^k+1-z^k_relation}}{=} A_k\la\nabla \Psi(\Blm^{k+1},\Bxi^{k+1}),\y^k - \Blm^{k+1}\ra + \alpha_{k+1}\la\nabla \Psi(\Blm^{k+1},\Bxi^{k+1}),\Bzeta^k - \Blm \ra\\ 
    &\hspace{-2.55cm}\overset{\eqref{eq:inner_prod_bound_1},\eqref{eq:inner_prod_bound_2}}{\le} A_k\left(\psi(\y^k) - \psi(\Blm^{k+1})\right) + A_k\la\nabla\Psi(\Blm^{k+1},\Bxi^{k+1}) - \nabla\psi(\Blm^{k+1}),\y^{k} - \Blm^{k+1} \ra\\
    &\hspace{-3cm}+A_{k+1}\left(\psi(\Blm^{k+1}) - \psi(\y^{k+1})\right) + \frac{1}{2}\|\Bzeta^k - \Blm\|_2^2 - \frac{1}{2}\|\Bzeta^{k+1}-\Blm\|_2^2\notag\\
    &\hspace{-5.9cm} + \frac{A_{k+1}}{2L_\psi}\left\|\nabla\Psi(\Blm^{k+1},\Bxi^{k+1}) - \nabla\psi(\Blm^{k+1})\right\|_2^2.
\end{eqnarray*}
Rearranging the terms and using $A_{k+1} = A_k + \alpha_{k+1}$, we obtain
\begin{eqnarray*}
    A_{k+1}\psi(\y^{k+1}) - A_k\psi(\y^{k}) &\le&  \alpha_{k+1}\left(\psi(\Blm^{k+1}) + \la\nabla \Psi(\Blm^{k+1},\Bxi^{k+1}), \Blm - \Blm^{k+1}\ra\right) + \frac{1}{2}\|\Bzeta^k - \Blm\|_2^2\\
    &&\quad - \frac{1}{2}\|\Bzeta^{k+1}-\Blm\|_2^2 + \frac{A_{k+1}}{2L_\psi}\left\|\nabla\Psi(\Blm^{k+1},\Bxi^{k+1}) - \nabla\psi(\Blm^{k+1})\right\|_2^2\\
    &&\quad + A_k\la\nabla\Psi(\Blm^{k+1},\Bxi^{k+1}) - \nabla\psi(\Blm^{k+1}),\y^{k} - \Blm^{k+1} \ra,
\end{eqnarray*}
and after summing these inequalities for $k=0,\ldots,N-1$ we get
\begin{eqnarray*}
       A_N\psi(\y^N)  &\leq& \frac{1}{2}\|\Blm - \Bzeta^0\|_2^2 - \frac{1}{2}\|\Blm - \Bzeta^N\|_2^2 + \sum_{k=0}^{N-1} \alpha_{k+1} \left( \psi(\Blm^{k+1}) + \la \nabla \Psi(\Blm^{k+1}, \Bxi^{k+1}), \Blm - \Blm^{k+1}\ra \right)  \notag \\ 
 &&\quad +  \sum_{k=0}^{N-1}A_{k}\la \nabla \Psi(\Blm^{k+1}, \Bxi^{k+1})  - \nabla \psi(\Blm^{k+1}) , \y^k- \Blm^{k+1} \ra \notag\\
 &&\quad +  \sum_{k=0}^{N-1}\frac{A_{k+1}}{2L_\psi}\|\nabla \psi(\Blm^{k+1}) -\nabla \Psi(\Blm^{k+1}, \Bxi^{k+1})\|_{2}^2,
 \end{eqnarray*}
 where we use that $A_0 = 0$.
\end{proof}

Now, we are ready to prove our main result in Theorem~\ref{Th:stoch_err} on the communication and oracle complexity of Algorithm~\ref{Alg:DualStochAlg}. For convenience we provide the statement of the theorem once again.

\begin{theorem}
\label{Th:stoch_err_appendix}
Assume that $F$ is $\mu$-strongly convex and  $\|\nabla F(\x^*)\|_2 = M_F$. Let
     $\varepsilon>0$ be a desired accuracy. Assume that at each iteration of Algorithm \ref{Alg:DualStochAlg} the approximation for $\nabla \psi(\y)$ is chosen according to \eqref{eq:batched_estimates} with batch size  $r_k = \pd{\Omega\big(\max \big\{ 1, \nicefrac{ \sigma^2_\psi {\alpha}_k \ln(\nicefrac{N}{\delta})}{\e}\big\}\big) }$. \ed{Assume additionally that $F$ is $L_F$-Lipschitz continuous on the set $B_{R_F}(0) = \{\x\in \R^{nm}\mid \|\x\|_2 \le R_F\}$ where $R_F = {\Omega}\left(\max\left\{\frac{R_\y}{A_N}\sqrt{\frac{6C_2H}{\lambda_{\max}(W)}}, \frac{\lambda_{\max}(\sqrt{W})JR_\y}{\mu}, R_\x\right\}\right)$, $R_\y$ is such that $\|\y^*\|_2 \leq R_\y$, $\y^*$ being an optimal solution of the dual problem and $R_\x = \|\x(\sqrt{W}\y^*)\|_2$.}
Then, after  $N = \widetilde{O}\left(\sqrt{(\nicefrac{M^2_{F}}{\mu\e})\chi(W)} \right)$ iterations, the outputs $\x^N$ and $\y^N$ of Algorithm \ref{Alg:DualStochAlg} satisfy 
\begin{align}
    F(\x^N) - F(\x^*) \leq \e, \quad \pd{\|\sqrt{W}\x^N\|_2 \leq \frac{\e}{R_\y}}
\end{align}
with probability at least $ 1-\ed{4}\delta$, \pd{where $\delta \in \left(\pd{0},\nicefrac{1}{\ed{4}}\right)$}, \eg{$\ln(\nicefrac{N}{\delta}) \ge 3$}.
Moreover, the number of stochastic oracle calls for the dual function $\nabla \vp_i$ per node $i=1,\dots m$ is
{
\begin{align*}
    O\left(\max\left\{  \frac{\sigma_{\psi}^2 M_F^2}{\e^2\lm_{\min}^+(W)} \ln\left(\frac{1}{\delta}\sqrt{\frac{M^2_F}{\mu\e}\chi(W)} \right), \sqrt {\frac{M^2_F}{\mu\e}\chi(W)} \right\} \right)
\end{align*}
}
\end{theorem}
\begin{proof}
From Lemma~\ref{lem:psi_y_N_bound} we have
\begin{eqnarray}
     A_N\psi(\y^N)  &\leq& \frac{1}{2}\|\Blm - \Bzeta^0\|_2^2 - \frac{1}{2}\|\Blm - \Bzeta^N\|_2^2 + \sum\limits\limits_{k=0}^{N-1} \alpha_{k+1} \left( \psi(\Blm^{k+1}) + \la \nabla \Psi(\Blm^{k+1}, \Bxi^{k+1}), \Blm - \Blm^{k+1}\ra \right)  \notag \\ 
 &&\quad+  \sum\limits_{k=0}^{N-1}A_{k}\la   \nabla \Psi(\Blm^{k+1}- \nabla \psi(\Blm^{k+1}, \Bxi^{k+1}), \y^k- \Blm^{k+1} \ra \notag \\
 &&\quad+  \sum\limits_{k=0}^{N-1}\frac{A_{k+1}}{2\pd{L_\psi}}\|\nabla \psi(\Blm^{k+1}) -\nabla \Psi(\Blm^{k+1}, \Bxi^{k+1})\|_{2}^2.\label{eq:stoch_primal_dual_cond}
 \end{eqnarray}

From definition of $\Blm^{k+1}$ (see \eqref{eq:Alg_lambda}) we have
\begin{eqnarray}\label{eq:y_blm}
    \alpha_{k+1}\left(\Blm^{k+1} - \Bzeta^k\right) = A_k\left(\y^k - \Blm^{k+1}\right).
\end{eqnarray}
Using this, we add and subtract $\sum_{k=0}^{N-1}\alpha_{k+1}\la \nabla\psi(\Blm^{k+1}), \Blm^* - \Blm^{k+1}\ra$ in \eqref{eq:stoch_primal_dual_cond}, and obtain
by choosing $\Blm = \Blm^*$
 \begin{eqnarray}
A_N\psi(\y^N)  &\leq& \frac{1}{2}\|\Blm^* - \Bzeta^0\|_2^2 - \frac{1}{2}\|\Blm^* - \Bzeta^N\|_2^2 + \sum_{k=0}^{N-1} \alpha_{k+1} \left( \psi(\Blm^{k+1}) + \la \nabla \psi(\Blm^{k+1}), \Blm^* - \Blm^{k+1}\ra \right)  \notag \\ 
&& +  \sum_{k=0}^{N-1}\alpha_{k+1}\la \nabla \Psi(\Blm^{k+1}, \Bxi^{k+1}) - \nabla \psi(\Blm^{k+1}), \a^k\ra \notag \\
&& +  \sum_{k=0}^{N-1}\alpha_{k+1}^2\|\nabla \psi(\Blm^{k+1}) -\nabla \Psi(\Blm^{k+1}, \Bxi^{k+1})\|_{2}^2,\label{eq:stoch_primal_dual_cond2}
\end{eqnarray}
where $\a^k = \Blm^* - \Bzeta^k$.
From convexity of $\psi$ we have 
\begin{eqnarray}
 \sum_{k=0}^{N-1} \alpha_{k+1} \left( \psi(\Blm^{k+1}) + \la \nabla \psi(\Blm^{k+1}), \Blm^* - \Blm^{k+1}\ra \right) &\le& \sum_{k=0}^{N-1} \alpha_{k+1} \left( \psi(\Blm^{k+1}) + \psi(\Blm^{*}) - \psi(\Blm^{k+1}) \right) \notag\\
&=& \psi(\Blm^{*})\sum_{k=0}^{N-1}\alpha_{k+1} = A_N\psi(\Blm^{*}) \le A_N\psi(\y^{N})  \notag
\end{eqnarray}

From this and \eqref{eq:stoch_primal_dual_cond2} we get
\begin{eqnarray}
 \frac{1}{2}\|\Blm^* - \Bzeta^N\|_2^2 &\overset{\eqref{eq:stoch_primal_dual_cond2}}{\le}& \frac{1}{2}\|\Blm^* - \Bzeta^0\|_2^2 +  \sum_{k=0}^{N-1}\alpha_{k+1}\la \nabla \Psi(\Blm^{k+1}, \Bxi^{k+1}) - \nabla \psi(\Blm^{k+1}), \a^k\ra \notag \\
&&\quad +  \sum_{k=0}^{N-1}\alpha_{k+1}^2\|\nabla \psi(\Blm^{k+1}) -\nabla \Psi(\Blm^{k+1}, \Bxi^{k+1})\|_{2}^2.\label{eq:stoch_primal_dual_cond3}
\end{eqnarray}

Next step we introduce sequences $\{R_k\}_{k\ge 0}$ and $\{\widetilde{R}_k\}_{k\ge 0}$ as follows
\begin{align*}
    R_k = \|\Bzeta_k - \Blm^*\|_2  \text{ and }  \widetilde{R}_k = \max\left\{\widetilde{R}_{k-1}, R_k\right\}, \widetilde{R}_0 = R_0.
\end{align*}
Since $\Bzeta^0=0$ in Algorithm \ref{Alg:DualStochAlg}, then $R_0 = R_\y$, where $R_\y$ is such that $\|\Blm^*\|_2 \leq R_\y$.
One can obtain by induction that  $\forall k\geq 0 ~\Blm^{k+1},\y^k,\Bzeta^k\in B_{\widetilde{R}_{k}}(\Blm^*) $, where  $B_{\widetilde{R}_{k}}(\Blm^*)$ is Euclidean ball with radius $\widetilde{R}_{k}$ \pd{and center} $\Blm^*$. Indeed, since from \eqref{eq:Alg_y} \pd{$\y^{k+1}$}
is a convex combination of $\pd{\Bzeta^{k+1}}\in B_{R_{k+1}}(\Blm^*) \subseteq B_{\widetilde{R}_{k+1}}(\Blm^*)$ and $\y^k\in B_{\widetilde{R}_k}(\Blm^*)\subseteq B_{\widetilde{R}_{k+1}}(\Blm^*)$, where we use the fact that a ball is a convex set, we get $\y^{k+1}\in B_{\widetilde{R}_{k+1}}(\Blm^*)$. Analogously, since from \eqref{eq:Alg_lambda} $\Blm^{k+1} $
is a convex combination of $\y^k$ and $\Bzeta^k$ we have $\Blm^{k+1}\in B_{\widetilde{R}_{k}}(\Blm^*)$. Using new notation we can rewrite \eqref{eq:stoch_primal_dual_cond3} as  
\begin{eqnarray}
     \frac{1}{2}R_N^2 &\le& \frac{1}{2}R_\y^2 + \sum\limits_{k=0}^{N-1}\alpha_{k+1}\la \nabla \Psi(\Blm^{k+1}, \Bxi^{k+1}) - \nabla \psi(\Blm^{k+1}), \a^k\ra \notag \\
 &&\quad +  \sum\limits_{k=0}^{N-1}\alpha_{k+1}^2\|\nabla \psi(\Blm^{k+1}) -\nabla \Psi(\Blm^{k+1}, \Bxi^{k+1})\|_{2}^2,\label{eq:radius_for_prima_dual}
\end{eqnarray}
where $\|\a^k\|_2 = \|\Blm^* - \Bzeta^k\|_2 \le \widetilde{R}_k$.

Let us denote $\eta^{k+1} = \nabla \Psi(\Blm^{k+1}, \Bxi^{k+1}) - \nabla \psi(\Blm^{k+1}) 
$. Theorem~2.1 from \cite{juditsky2008large} (see Lemma~\ref{lem:jud_nem_large_dev} in the Section~\ref{sec:aux_results}) says that
{\begin{align*}
&\PP\left\{\|\eta^k\|_2 \ge \left(\sqrt{2} + \sqrt{2}\gamma\right)\sqrt{\frac{\sigma_\psi^2}{r_{k+1}}}\mid\{\eta^j\}_{j=0}^{k-1} \right\} \le \exp\left(-\frac{\gamma^2}{3}\right).
\end{align*}}
Using this and Lemma~2 from \cite{jin2019short} (see Lemma~\ref{lem:jin_lemma_2} in the Section~\ref{sec:aux_results}) we get that $\EE\left[\exp\left({\frac{\|\eta^k\|_2^2}{\sigma_k^2}}\right)|\{\eta^j\}_{j=0}^{k-1}\right] \le \exp(1)$, where $\sigma_k^2 \le \frac{\widetilde{C}\sigma_\psi^2}{r_{k+1}} \le \frac{C\varepsilon}{\widetilde{\alpha}_{k+1}\ln(\frac{N}{\delta})}$, \pdd{where $\widetilde{\alpha}_{k+1}$ is defined in \eqref{eq:alpha_estimate_appendix}}, $\widetilde{C}$ and $C$ are some positive constants.  Moreover, $\a^k$ depends only on $\eta^0,\ldots,\eta^{k-1}$. Putting all together in \eqref{eq:radius_for_prima_dual} and changing the indices we get, \pd{for all $l =1,...,N$,}
\begin{equation*}
    \frac{1}{2}R_l^2 \le \frac{1}{2}R_\y^2 + \sum\limits_{k=0}^{l-1}\alpha_{k+1}\la\eta^{k+1},\a^k\ra + \sum\limits_{k=0}^{l-1}\alpha_{k+1}\|\eta^{k+1}\|_2^2.
\end{equation*}
Next we apply the Lemma~\ref{lem:tails_estimate} with the constants $A = \frac{1}{2}R_0^2, u = 1, c = 1, D = \frac{1}{2L}, d = 1$ and using $\varepsilon \le \frac{HLR_0^2}{N^2}$ which holds for some positive constant $H$ due to our choice of $N$, and get that with probability at least $1-2\delta$ the inequalities
\begin{eqnarray}
    \widetilde{R}_l &\le& JR_\y \quad \text{ and } \label{eq:bounding_tilde_R_l} \\
    \sum\limits_{k=0}^{l-1}\alpha_{k+1}\la\eta^k,\a^k \ra + \sum\limits_{k=0}^{l-1}\alpha_{k+1}^2\|\eta^k\|_2^2 &\le& \left(12CH + C_1\sqrt{\frac{CHJg(N)}{2}}\right)R_\y^2,\label{eq:stoch_terms_estimation}
\end{eqnarray}
hold for all $l=1,\ldots,N$ simultaneously, where $C_1$ is some positive constant, $g(N) = \frac{\ln\left(\frac{N}{\delta}\right) + \ln\ln\left(\frac{B}{b}\right)}{\ln\left(\frac{N}{\delta}\right)}$, $B = CHR_0^2\left(2R_0^2 + \frac{18C}{L}\varepsilon N(N+3)\right)2^N$, $b = \sigma_0^2\widetilde{\alpha}_{1}^2R_0^2$ and $$J = \max\left\{1, C_1\sqrt{\frac{CH g(N)}{2}} + \sqrt{\frac{C_1^2CH g(N)}{2} + 1 + 24CH}\right\}.
$$

To estimate the duality gap we need again refer to \eqref{eq:stoch_primal_dual_cond}. Since $\Blm$ is chosen arbitrary we can take the minimum in $\Blm$ by the set $B_{2R_\y}(0) = \{\Blm: \|\Blm\|_2\leq 2R_\y\}$
{
\begin{eqnarray}
A_N\psi(\y^N)  &\leq&  \min\limits_{\Blm \in B_{\pd{2}R_\y}(0)} \left\{ \frac{1}{2}\|\Blm - \Bzeta^0\|_2^2  + \sum\limits_{k=0}^{N-1} \alpha_{k+1} \left( \psi(\Blm^{k+1})  + \la \nabla \Psi(\Blm^{k+1}, \Bxi^{k+1}), \Blm - \Blm^{k+1}\ra \right) \right\} \notag \notag \\ 
 &&\quad +  \sum_{k=0}^{N-1}A_{k}\la \nabla \Psi(\Blm^{k+1}, \Bxi^{k+1})-\nabla \psi(\Blm^{k+1}), \pd{\y^k - \Blm^{k+1} } \ra \notag \\
 &&\quad  +  \sum_{k=0}^{N-1}\frac{A_{k}}{2\pd{L_\psi}}\|\nabla \psi(\Blm^{k+1}) -\nabla \Psi(\Blm^{k+1}, \Bxi^{k+1})\|_{2}^2\notag \notag \\
      &\overset{\eqref{eq:y_blm}}{\le}& 2R_\y^2 + \min\limits_{\Blm \in B_{\pd{2}R_\y}(0)}  \sum_{k=0}^{N-1} \alpha_{k+1} \left( \psi(\Blm^{k+1}) + \la \nabla \Psi(\Blm^{k+1}, \Bxi^{k+1}), \Blm - \Blm^{k+1}\ra \right)  \notag \\ 
 &&\quad+  \sum_{k=0}^{N-1}\alpha_{k+1}\la \nabla \Psi(\Blm^{k+1}, \Bxi^{k+1})-\nabla \psi(\Blm^{k+1}), \Blm^{k+1} - \Bzeta^k\ra \notag \\
 &&\quad +  \sum\limits_{k=0}^{N-1}\frac{A_{k+1}}{2\pd{L_\psi}}\|\nabla \psi(\Blm^{k+1}) -\nabla \Psi(\Blm^{k+1}, \Bxi^{k+1})\|_{2}^2,\label{eq:stoch_primal_dual_cond_min}
\end{eqnarray}
}  where we also used $\frac{1}{2}\|\Blm-\Bzeta^N\|^2_2\geq 0$ and $\Bzeta_0=0$. By adding and subtracting $\sum_{k=0}^{N-1}\alpha_{k+1}\la \nabla\psi(\Blm^{k+1}), \Blm^* - \Blm^{k+1}\ra$ under minimum in \eqref{eq:stoch_primal_dual_cond_min} we obtain
\begin{eqnarray}
     \min\limits_{\Blm \in B_{2R_\y}(0)}  \sum\limits_{k=0}^{N-1} \alpha_{k+1} \left( \psi(\Blm^{k+1}) + \la \nabla \Psi(\Blm^{k+1}, \Bxi^{k+1}), \Blm - \Blm^{k+1}\ra \right)& \notag \\
     &\hspace{-3cm} \le \min\limits_{\Blm \in B_{2R_\y}(0)}  \sum_{k=0}^{N-1} \alpha_{k+1} \left( \psi(\Blm^{k+1})+ \la \nabla \psi(\Blm^{k+1}), \Blm - \Blm^{k+1}\ra \right) \notag\\
     &\hspace{-3cm}+ \max\limits_{\Blm\in B_{2R_\y}(0)}\sum\limits_{k=0}^{N-1} \alpha_{k+1}\la\nabla \Psi(\Blm^{k+1}, \Bxi^{k+1}) - \nabla\psi(\Blm^{k+1}), \Blm\ra \notag \\
     &\hspace{-3.62cm}+ \sum\limits_{k=0}^{N-1} \alpha_{k+1}\la\nabla \Psi(\Blm^{k+1}, \Bxi^{k+1}) - \nabla\psi(\Blm^{k+1}), -\Blm^{k+1}\ra\notag.
\end{eqnarray}
Since $-\Blm^*\in B_{2R_\y}(0)$ we have that
\begin{eqnarray}
      \sum\limits_{k=0}^{N-1} \alpha_{k+1}\la\nabla \Psi(\Blm^{k+1}, \Bxi^{k+1}) - \nabla\psi(\Blm^{k+1}), -\Blm^{k+1}\ra &=& \sum\limits_{k=0}^{N-1} \alpha_{k+1}\la\nabla \Psi(\Blm^{k+1}, \Bxi^{k+1}) - \nabla\psi(\Blm^{k+1}), \Blm^*-\Blm^{k+1}\ra \notag \\
      &&\quad + \sum\limits_{k=0}^{N-1} \alpha_{k+1}\la\nabla \Psi(\Blm^{k+1}, \Bxi^{k+1}) - \nabla\psi(\Blm^{k+1}), -\Blm^{*}\ra \notag \\
      &\le& \sum\limits_{k=0}^{N-1} \alpha_{k+1}\la\nabla \Psi(\Blm^{k+1}, \Bxi^{k+1}) - \nabla\psi(\Blm^{k+1}), \Blm^*-\Blm^{k+1}\ra \notag \\
      &&\quad + \max\limits_{\Blm\in B_{2R_\y}(0)}\sum\limits_{k=0}^{N-1} \alpha_{k+1}\la\nabla \Psi(\Blm^{k+1}, \Bxi^{k+1}) - \nabla\psi(\Blm^{k+1}), \Blm\ra.\notag
\end{eqnarray}
Putting all together in \eqref{eq:stoch_primal_dual_cond_min} and using \eqref{eq:Alg_const} we get
{
\begin{eqnarray}
    A_N\psi(\y^N)  &\leq& 2R_\y^2 + \min\limits_{\Blm \in B_{2R_\y}(0)}  \sum\limits_{k=0}^{N-1} \alpha_{k+1} \left( \psi(\Blm^{k+1}) +  \la \nabla \psi(\Blm^{k+1}), \Blm - \Blm^{k+1}\ra \right)  \notag \\ 
  &&\quad+ 2\max\limits_{\Blm\in B_{2R_\y}(0)}\sum\limits_{k=0}^{N-1} \alpha_{k+1}\la\nabla \Psi(\Blm^{k+1}, \Bxi^{k+1}) - \nabla\psi(\Blm^{k+1}), \Blm\ra\notag\\
 &&\quad+  \sum_{k=0}^{N-1}\alpha_{k+1}\la \nabla \Psi(\Blm^{k+1}, \Bxi^{k+1})-\nabla \psi(\Blm^{k+1}), \a^k\ra \notag \\
 &&\quad+  \sum_{k=0}^{N-1}\alpha_{k+1}^2\|\nabla \psi(\Blm^{k+1}) -\nabla \Psi(\Blm^{k+1}, \Bxi^{k+1})\|_{2}^2,\label{eq:stoch_primal_dual_cond_min2}
\end{eqnarray}
} \noindent
where $\a^k = \Blm^* - \Bzeta^k$. From \eqref{eq:stoch_terms_estimation} we have that with probability at least $1-2\delta$ the following inequality holds:
{
\begin{eqnarray}
    A_N\psi(\y^N)  &\leq& \min\limits_{\Blm \in B_{2R_\y}(0)}  \sum_{k=0}^{N-1} \alpha_{k+1} \left( \psi(\Blm^{k+1}) + \la \nabla \psi(\Blm^{k+1}), \Blm - \Blm^{k+1}\ra \right)  \notag \\ 
  &&\quad+ 2\max\limits_{\Blm\in B_{2R_\y}(0)}\sum\limits_{k=0}^{N-1} \alpha_{k+1}\la\nabla \Psi(\Blm^{k+1}, \Bxi^{k+1}) {-}\nabla\psi(\Blm^{k+1}), \Blm\ra \notag\\
 &&\quad+ 2R_\y^2 {+}  \left(12CH + C_1\sqrt{\frac{CHJg(N)}{2}}\right)R^2_\y.\label{eq:stoch_primal_dual_cond_min3}
\end{eqnarray}
}
By the definition of the norm we get
{
\begin{eqnarray}
      \max\limits_{\Blm\in B_{2R_\y}(0)}\sum\limits_{k=0}^{N-1} \alpha_{k+1}\la\nabla \Psi(\Blm^{k+1}, \Bxi^{k+1}) - \nabla\psi(\Blm^{k+1}), \Blm\ra \leq  2R_{\y}\left\|\sum\limits_{k=0}^{N-1}\alpha_{k+1}(\nabla \Psi(\Blm^{k+1}, \Bxi^{k+1}){-} \nabla \psi(\Blm^{k+1})\right\|_2.\label{eq:maximum_estimate_first_step}
\end{eqnarray}
}
Next we apply Lemma~\ref{lem:jud_nem_large_dev} to the r.h.s of previous inequality and get
\begin{eqnarray}
    \PP\left\{\left\|\sum_{k=0}^{N-1}\alpha_{k+1}(\nabla \Psi(\Blm^{k+1}, \Bxi^{k+1}) - \nabla \psi(\Blm^{k+1})\right\|_2 \ge \left(\sqrt{2} + \sqrt{2}\gamma\right)\sqrt{\sum\limits_{k=0}^{N-1}\alpha_{k+1}^2\frac{\sigma_\psi^2}{r_{k+1}}}\right\} \le \exp\left(-\frac{\gamma^2}{3}\right). \notag 
\end{eqnarray}
Since $N^2\le \frac{HL_\psi R_0^2}{\e}$ and $r_k = \Omega\left(\max\left\{1,\frac{ \sigma^2_\psi {\alpha}_k \ln(N/\delta)}{\e}\right\}\right)$ one can choose such $C_2 > 0$ that $\frac{\sigma_\psi^2}{r_k} \le \frac{C_2\varepsilon}{\alpha_k\ln\left(\frac{N}{\delta}\right)} \le \frac{HL_\psi C_2R_0^2}{\alpha_k N^2\ln\left(\frac{N}{\delta}\right)}$. Let us choose $\gamma$ such that $\exp\left(-\frac{\gamma^2}{3}\right) = \delta:$ $\gamma = \sqrt{3\ln(\nicefrac{1}{\delta})}$. From this we get that \pd{with} probability at least $1-\delta$
\begin{eqnarray}
\left\|\sum_{k=0}^{N-1}\alpha_{k+1}(\nabla \Psi(\Blm^{k+1}, \Bxi^{k+1}) - \nabla \psi(\Blm^{k+1}))\right\|_2  &\le& \sqrt{2}\left(1 + \sqrt{\ln\frac{1}{\delta}}\right)R_y\sqrt{\frac{HL_\psi C_2}{\ln\left(\frac{N}{\delta}\right)}}\sqrt{\sum\limits_{k=0}^{N-1}\frac{\alpha_{k+1}}{N^2}}\notag\\
    &\overset{\eqref{eq:alpha_estimate}}{\le}& 2\sqrt{2}R_y\sqrt{HL_\psi C_2}\sqrt{\sum\limits_{k=0}^{N-1}\frac{k+2}{2L_\psi N^2}}
    = 2R_y\sqrt{HC_2}\sqrt{\frac{N(N+3)}{N^2}}\notag\\
    &\le& 4R_\y\sqrt{{C_2}}\label{eq:maximum_estimate}.
\end{eqnarray}

Putting all together and using union bound we get that with probability at least $1-3\delta$
\begin{eqnarray}
    A_N\psi(\y^N)  
    &\overset{\eqref{eq:stoch_primal_dual_cond_min3}+\eqref{eq:maximum_estimate_first_step}+\eqref{eq:maximum_estimate}}{\le}& \min\limits_{\Blm \in B_{2R_\y}(0)}  \sum\limits_{k=0}^{N-1} \alpha_{k+1} \left( \psi(\Blm^{k+1}) + \la \nabla \psi(\Blm^{k+1}), \Blm - \Blm^{k+1}\ra \right)  \notag \\ 
    &&\quad+\left(8 \sqrt{HC_2}+  2 + 12CH + C_1\sqrt{\frac{CHJg(N)}{2}}\right)R_y^2.\label{eq:pr_dual_pre_final_estimate}
\end{eqnarray}

\ed{This brings us to the final part of the proof. Firstly, by definition of $\psi(\Blm^k)$ and Demyanov--Danskin's theorem we have
\begin{eqnarray}
    \psi(\Blm^k) - \la\nabla \psi(\Blm^k),\Blm^k \ra &=& \la\Blm^k, \sqrt{W}\x(\sqrt{W}\Blm^k) \ra - F(\x(\sqrt{W}\Blm^k)) - \la\nabla \psi(\Blm^k),\Blm^k \ra \notag\\
    &=& - F(\x(\sqrt{W}\Blm^k)).\notag
\end{eqnarray}
Summing up this equality for $k=1,\ldots, N$ with weights $\alpha_{k}$ and using convexity of $F$ we get
\begin{eqnarray}
    \sum\limits_{k=0}^{N-1}\alpha_{k+1}(\psi(\Blm^{k+1}) - \la\nabla\psi(\Blm^{k+1}), \Blm^{k+1}\ra) &=& -A_N\sum\limits_{k=0}^{N-1}\frac{\alpha_{k+1}}{A_N}F(\x(\sqrt{W}\Blm^{k+1}))\notag\\
    &\le& -A_NF\left(\sum\limits_{k=0}^{N-1}\frac{\alpha_{k+1}}{A_N}\x(\sqrt{W}\Blm^{k+1})\right) = -A_NF(\hat \x_N),\label{eq:pr_dual_Lend_of_the_proof_1}
\end{eqnarray}
where $\hat \x_N \eqdef \frac{1}{A_N}\sum_{k=0}^{N-1}\alpha_{k+1}\x(\sqrt{W}\Blm^{k+1})$.
Secondly, by definition of the norm
\begin{eqnarray}
    \min\limits_{\Blm\in B_{2R_{\y}}(0)}\sum\limits_{k=0}^{N-1}\alpha_{k+1}\left\la\nabla\psi(\Blm^{k+1}), \Blm\right\ra &=& \min\limits_{\Blm\in B_{2R_{\y}}(0)}\left\la\sum\limits_{k=0}^{N-1}\alpha_{k+1}\nabla\psi(\Blm^{k+1}), \Blm\right\ra\notag\\
    &=&-2R_{\y}A_N\left\|\frac{1}{A_N}\sum\limits_{k=0}^{N-1}\alpha_{k+1}\nabla\psi(\Blm^{k+1})\right\|_2\notag\\
    &=&-2R_{\y}A_N\left\|\frac{1}{A_N}\sum\limits_{k=0}^{N-1}\alpha_{k+1}\sqrt{W}\x(\sqrt{W}\Blm^{k+1})\right\|_2\notag\\
    &=&-2R_{\y}A_N\|\sqrt{W}\hat\x^N\|_2.\label{eq:pr_dual_Lend_of_the_proof_2}
\end{eqnarray}}

\ed{Combining inequalities \eqref{eq:pr_dual_pre_final_estimate}, \eqref{eq:pr_dual_Lend_of_the_proof_1} and \eqref{eq:pr_dual_Lend_of_the_proof_2} we obtain that with probability at least $1 - 3\delta$
\begin{eqnarray}
    A_N\psi(\y^N) &\overset{\eqref{eq:pr_dual_pre_final_estimate}}{\le}& \sum\limits_{k=0}^{N-1}\alpha_{k+1}(\psi(\Blm^{k+1}) - \la\nabla\psi(\Blm^{k+1}), \Blm^{k+1}\ra) + \min\limits_{\Blm\in B_{2R_{\y}}(0)}\sum\limits_{k=0}^{N-1}\alpha_{k+1}\left\la\nabla\psi(\Blm^{k+1}), \Blm\right\ra\notag\\
    &&\quad + \left(8 \sqrt{HC_2}+  2 + 12CH + C_1\sqrt{\frac{CHJg(N)}{2}}\right)R_y^2\notag\\
    &\overset{\eqref{eq:pr_dual_Lend_of_the_proof_1} + \eqref{eq:pr_dual_Lend_of_the_proof_2}}{\le}& -A_NF(\hat \x^N) - 2R_\y A_N\|\sqrt{W}\hat\x_N\|_2 + \left(8 \sqrt{HC_2}+  2 + 12CH + C_1\sqrt{\frac{CHJg(N)}{2}}\right)R_y^2.\label{eq:pr_dual_almost_finish}
\end{eqnarray}}

\ed{Lemma~\ref{lem:jud_nem_large_dev} states that for all $\gamma > 0$
\begin{equation*}
    \PP\left\{\left\|\sum_{k=0}^{N-1}\alpha_{k+1}\left(\x(\sqrt{W}\Blm^{k+1},\Bxi^{k+1}) - \x(\sqrt{W}\Blm^{k+1})\right)\right\|_2 \ge (\sqrt{2} + \sqrt{2}\gamma)\sqrt{\sum\limits_{k=0}^{N-1}\frac{\alpha_{k+1}^2\sigma_{\x}^2}{r_{k+1}}}\right\} \le \exp\left(-\frac{\gamma^2}{3}\right).
\end{equation*}
Taking $\gamma = \sqrt{3\ln\frac{1}{\delta}}$ and using $r_k \ge \frac{\sigma_\psi^2\alpha_k\ln\frac{N}{\delta}}{C_2\varepsilon}$ we get that with probability at least $1 - \delta$
\begin{eqnarray}
    \|\x^N - \hat\x^N\|_2 &=& \frac{1}{A_N}\left\|\sum\limits_{k=0}^{N-1}\alpha_{k+1}\left(\x(\sqrt{W}\Blm^{k+1},\Bxi^{k+1}) - \x(\sqrt{W}\Blm^{k+1})\right)\right\|_2\notag\\
    &\le& \frac{\sqrt{2}}{A_N}\left(1 + \sqrt{3\ln\frac{1}{\delta}}\right)\sqrt{\sum\limits_{k=0}^{N-1}\frac{\alpha_{k+1}^2\sigma_\x^2}{r_{k+1}^2}}\notag\\
    &\le& \frac{2}{A_N}\sqrt{6\ln\frac{1}{\delta}}\frac{1}{\sqrt{\ln\frac{N}{\delta}}}\sqrt{\sum\limits_{k=0}^{N-1}\frac{C_2\alpha_{k+1}\varepsilon}{\lambda_{\max}(W)}}\notag\\
    &\le& \frac{2}{A_N}\sqrt{\frac{6C_2}{\lambda_{\max}(W)}}\sqrt{\sum\limits_{k=0}^{N-1}\frac{(k+2)HL_\psi R_\y^2}{2L_\psi N^2}} \le \frac{2R_\y}{A_N}\sqrt{\frac{6C_2H}{\lambda_{\max}(W)}}.\label{eq:pr_dual_Lend_of_the_proof_3}
\end{eqnarray}
It implies that with probability at least $1 - \delta$
\begin{eqnarray}
    \|\sqrt{W}\x^N - \sqrt{W}\hat\x^N\|_2 &\le& \|\sqrt{W}\|_2 \cdot \|\x^N - \hat\x^N\|_2\notag\\
    &\overset{\eqref{eq:pr_dual_Lend_of_the_proof_3}}{\le}& \sqrt{\lambda_{\max}(W)}\frac{2R_\y}{A_N}\sqrt{\frac{6C_2H}{\lambda_{\max}(W)}} = \frac{2R_\y}{A_N}\sqrt{6C_2H}\label{eq:pr_dual_Lend_of_the_proof_4}
\end{eqnarray}
and due to triangle inequality with probability $\ge 1 - \delta$
\begin{eqnarray}
    2R_\y A_N\|\sqrt{W}\hat\x^N\|_2 &\ge& 2R_\y A_N\|\sqrt{W}\x^N\|_2 - 2R_\y A_N\|\sqrt{W}\hat\x^N - \sqrt{W}\x^N\|_2\notag\\
    &\overset{\eqref{eq:pr_dual_Lend_of_the_proof_4}}{\ge}& 2R_\y A_N \|\sqrt{W}\x^N\|_2 - 4R_\y^2\sqrt{6C_2H}.\label{eq:pr_dual_Lend_of_the_proof_5}
\end{eqnarray}}

\ed{Now we want to apply Lipschitz-continuity of $F$ on the ball $B_{R_F}(0)$ and specify our choice of $R_F$. Recall that $\x(\Blm) \eqdef \argmax_{\x\in \R^{nm}}\left\{\la\Blm, \x\ra - F(\x)\right\}$ and due to Demyanov-Danskin theorem $\x(\Blm) = \nabla\varphi(\Blm)$. Together with $L_\varphi$-smoothness of $\varphi$ it implies that
\begin{eqnarray}
    \|\x(\sqrt{W}\Blm^{k+1})\|_2 &=& \|\nabla \varphi(\sqrt{W}\Blm^{k+1})\|_2 \le \|\nabla \varphi(\sqrt{W}\Blm^{k+1}) - \nabla\varphi(\sqrt{W}\y^*)\|_2 + \|\nabla\varphi(\sqrt{W}\y^*)\|_2\notag\\
    &\le& L_\varphi\|\sqrt{W}\Blm^{k+1}-\sqrt{W}\y^*\|_2 + \|\x(\sqrt{W}\y^*)\|_2 \le \frac{\lambda_{\max}(\sqrt{W})}{\mu}\|\Blm^{k+1} - \y^*\|_2 + R_\x\notag
\end{eqnarray}
From this and \eqref{eq:bounding_tilde_R_l} we get that with probability at least $1-2\delta$ the inequality
\begin{eqnarray}
    \|\x(\sqrt{W}\Blm^{k+1})\|_2 &\overset{\eqref{eq:bounding_tilde_R_l}}{\le}&  \left(\frac{\lambda_{\max}(\sqrt{W})J}{\mu} + \frac{R_\x}{R_\y}\right)R_\y\label{eq:pr_dual_Lend_of_the_proof_6}
\end{eqnarray}
holds for all $k = 0,1,2,\ldots, N-1$ simultaneously since $\Blm^{k+1}\in B_{R_k}(\y^*) \subseteq B_{\widetilde{R}_{k+1}}(\y^*)$. Using the convexity of the norm we get that with probability at least $1-2\delta$
\begin{eqnarray}
    \|\hat\x^N\|_2 \le \frac{1}{A_N}\sum\limits_{k=0}^{N-1}\alpha_{k+1}\|\x(\sqrt{W}\Blm^{k+1})\|_2 \overset{\eqref{eq:pr_dual_Lend_of_the_proof_6}}{\le} \left(\frac{\lambda_{\max}(\sqrt{W})J}{\mu} + \frac{R_\x}{R_\y}\right)R_\y. \label{eq:pr_dual_Lend_of_the_proof_7}
\end{eqnarray}
We notice that the last inequality lies in the same probability event when \eqref{eq:bounding_tilde_R_l} holds.}

\ed{Consider the probability event $E = \{\text{inequalities } \eqref{eq:pr_dual_almost_finish}-\eqref{eq:pr_dual_Lend_of_the_proof_7} \text{ hold simultaneously}\}$. Using union bound we get that $\PP\{E\} \ge 1 - 4\delta$. Combining \eqref{eq:pr_dual_Lend_of_the_proof_3} and \eqref{eq:pr_dual_Lend_of_the_proof_7} we get that inequality
\begin{equation}
    \|\x^N\|_2 \le \|\x^N-\hat\x^N\|_2 + \|\hat\x^N\|_2 \le \left(\frac{2}{A_N}\sqrt{\frac{6C_2H}{\lambda_{\max}(W)}} + \frac{\lambda_{\max}(\sqrt{W})J}{\mu} + \frac{R_\x}{R_\y}\right)R_\y \label{eq:pr_dual_Lend_of_the_proof_8}
\end{equation}
lies in the event $E$. Here we can specify our choice of $R_F$: $R_F$ should \pdd{be} at least $\left(\frac{2}{A_N}\sqrt{\frac{6C_2H}{\lambda_{\max}(W)}} + \frac{\lambda_{\max}(\sqrt{W})J}{\mu} + \frac{R_\x}{R_\y}\right)R_\y$. Then we get that the fact that points $\x^N$ and $\hat\x^N$ lie in $B_{R_F}(0)$ is a consequence of $E$. Therefore, we can apply Lipschitz-continuity of $F$ for the points $\x^N$ and $\hat\x^N$ and get that inequalities
\begin{eqnarray}
    |F(\hat\x^N) - F(\x^N)| \le L_F\|\hat \x^N - \x^N\|_2 \overset{\eqref{eq:pr_dual_Lend_of_the_proof_3}}{\le} \frac{2L_FR_\y}{A_N}\sqrt{\frac{6C_2H}{\lambda_{\max}(W)}}\label{eq:pr_dual_Lend_of_the_proof_9}
\end{eqnarray}
and
\begin{eqnarray}
    A_NF(\hat\x^N) = A_NF(\x^N) + A_N\left(F(\hat\x^N) - F(\x^N)\right) \overset{\eqref{eq:pr_dual_Lend_of_the_proof_9}}{\ge} A_NF(\x^N) - 2L_FR_\y\sqrt{\frac{6C_2H}{\lambda_{\max}(W)}}\label{eq:pr_dual_Lend_of_the_proof_10}
\end{eqnarray}
also lie in the event $E$. It remains to use inequalities \eqref{eq:pr_dual_Lend_of_the_proof_5} and \eqref{eq:pr_dual_Lend_of_the_proof_10} to bound first and second terms in the right hand side of inequality \eqref{eq:pr_dual_almost_finish} and obtain that with probability at least $1-4\delta$
\begin{eqnarray}
    A_N\psi(\y^N) + A_N F(\x^N) +2R_\y A_N\|\sqrt{W}\x^N\|_2 &\le& \left(4\sqrt{6C_2H} + \frac{2L_F}{R_\y}\sqrt{\frac{6C_2H}{\lambda_{\max}(W)}} + 8 \sqrt{HC_2}\right.\notag\\
    &&\left.\quad\quad\quad\quad\quad + 2 + 12CH + C_1\sqrt{\frac{CHJg(N)}{2}}\right)R_y^2.\label{eq:pr_dual_Lend_of_the_proof_11}
\end{eqnarray}
Using that $A_N$ grows as $\Omega(N^2/\pd{L_\psi})$ \cite{nesterov2004introduction}, $L_\psi\leq \frac{\lm_{\max}(W)}{\mu}$ and, as in the Section~\ref{sec:proof_thm_3}, $R_\y \leq \frac{\|\nabla F(\x^*)\|_2^2}{\lm^+_{\min}(W)}$,  we obtain that the choice of $N$ in the theorem statement guarantees that the r.h.s. of the last inequality is no \pdd{greater} than \pdd{$\e A_N$}. 
\pdd{
By weak duality $-F(\x^*)\leq \psi(\y^*)$, \ed{we have with probability at least $1-4\delta$}
\begin{align}\label{eq:func_math_fin}
	F(\x^N)  - F(\x^*) \leq F(\x^N) + \psi(\y^*) \leq F(\x^N) + \psi(\y^N)  
    &\leq  \e.
\end{align}
Since $\y^*$ is an optimal solution of the dual problem, we have, for any $\x$, $F(\x^*)\leq F(\x) \ed{-} \la \y^*, \sqrt{W}\x \ra$.
Then using assumption $\|\y^*\|_2\leq R_{\y}$\ed{, Cauchy-Schawrz inequality $\la\y,\sqrt{W}\x\ra \ge -\|\y^*\|_2\cdot\|\sqrt{W}\x\|_2 = -R_\y\|\sqrt{W}\x\|_2$} and choosing $\x = \x^N$, we get
\begin{equation}\label{eq:weak_dual}
F(\x^N) \geq F(\x^*) - R_{\y}\|\sqrt{W}\x^N\|_2
\end{equation}
Using this and weak duality $-F(\x^*)\leq \psi(\y^*)$, we obtain
\begin{align*}
\psi(\y^N) +F(\x^N) \geq \psi(\y^*) +F(\x^N) \geq -F(\x^*) +F(\x^N) \geq -R_{\y}\|\sqrt{W}\x^N\|_2,
\end{align*}
\ed{which implies that inequality}
\begin{equation}\label{eq:MathIneq}
\|\sqrt{W}\x^N\|_2 \ed{\overset{\eqref{eq:pr_dual_Lend_of_the_proof_11}+\eqref{eq:func_math_fin}}{\le}} \frac{\e}{R_{\y}}
\end{equation}
\ed{holds together with \eqref{eq:func_math_fin} with probability at least $1-4\delta$.}}
Number of communication rounds is equal to the number of iterations similarly as for Algorithm \ref{Alg:DualNFGM}. The total number of stochastic gradient oracle calls is $\sum_{k=1}^Nr_k$, which gives the bound in the problem statement since $\sum_{k=1}^N\alpha_{k+1}=A_N$.}
\end{proof}

\end{document}